\documentclass[twoside,centertags                  %,draft
]{amsart}

\usepackage[cmtip,color,all,curve,
arc,poly
]{xy}

\usepackage{amssymb}
\usepackage{lscape}
\usepackage{graphicx}
\usepackage{mathrsfs}

\usepackage{float}

\usepackage[usenames,dvipsnames]{pstricks}
\usepackage{epsfig}
\usepackage{pst-grad} % For gradients
\usepackage{pst-plot} % For axes

\newtheorem{thm}[equation]{Theorem}
\newtheorem{cor}[equation]{Corollary}
\newtheorem{lem}[equation]{Lemma}
\newtheorem{prop}[equation]{Proposition}
\theoremstyle{definition}
\newtheorem{defn}[equation]{Definition}
\theoremstyle{remark}
\newtheorem{rem}[equation]{Remark}

\newcommand{\C}[1]{\mathscr{#1}}

\def\op{\mathrm{op}}
\newcommand{\operad}[1]{\mathrm{Op}(#1)}
\newcommand{\operadc}[1]{\mathrm{Op}^{c}(#1)}
\newcommand{\operadpc}[1]{\mathrm{Op}^{pc}(#1)}

\newcommand{\algebra}[2]{\mathrm{Alg}_{#2}(#1)}
\newcommand{\algebrac}[2]{\mathrm{Alg}^{c}_{#2}(#1)}
\newcommand{\algebrapc}[2]{\mathrm{Alg}^{pc}_{#2}(#1)}

\def\r{\rightarrow} % flecha -->
\def\l{\leftarrow} % flecha <--
\def\rr{\Rightarrow} % flecha ==>

\def\into{\rightarrowtail}
\def\onto{\twoheadrightarrow}

\newcommand{\id}[1]{u}

\def\dos{\mathbf{2}}

\newcommand{\Mod}[1]{\mathrm{Mod}(#1)}

\newcommand{\mor}[1]{{#1}^{\dos}}

\def\mon{\operatorname{Mon}}
\def\ho{\operatorname{Ho}}

\def\st{\stackrel} % abreviatura de \stackrel

\def\unit{\mathbb{I}} % abreviatura de \underline

\def\To{\longrightarrow}

\def\colim{\mathop{\operatorname{colim}}}

\newcommand{\val}[2]{\widetilde{#2}}

\newcommand{\card}{\operatorname{card}}

\numberwithin{equation}{subsection}

\newdir{ >}{{}*!/-3pt/@{>}}
\newdir{> }{{}*!/10pt/@{>}}
\newdir{>> }{{}*!/10pt/@{>>}}

\begin{document}

\title[Homotopy theory of non-symmetric operads II]
{Homotopy theory of non-symmetric operads II: change of base category and left properness}%
\author{Fernando Muro}%
\address{Universidad de Sevilla,
Facultad de Matem\'aticas,
Departamento de \'Algebra,
Avda. Reina Mercedes s/n,
41012 Sevilla, Spain\newline
\indent \textit{Home page}: \textnormal{\texttt{http://personal.us.es/fmuro}}}
\email{fmuro@us.es}

%\thanks{The author was partially supported
%by the Andalusian Ministry of Economy, Innovation and Science under the grant FQM-5713, by the Spanish Ministry of Education and
%Science under the MEC-FEDER grant  MTM2010-15831, and by the Government of Catalonia under the grant SGR-119-2009.}

\begin{abstract}
We prove, under mild assumptions, that a Quillen equivalence between symmetric monoidal model categories gives rise to a Quillen equivalence between their model categories of (non-symmetric) operads, and also between model categories of algebras over operads. We also show left properness results on model categories of operads and algebras over operads. As an application, we prove homotopy invariance for (unital) associative operads. 
\end{abstract}

\maketitle
%\tableofcontents

% ---------------------------------------------------------------------------------

\numberwithin{equation}{section}

\section{Introduction}

In this paper, we continue the study of the homotopy theory of non-symmetric operads started at \cite{htnso}. We deal with  two new topics: change of base category and left properness. 

Given a Quillen equivalence  $F\colon\C{V}\rightleftarrows\C{W}\colon G$ between symmetric monoidal model categories, we wish to have a Quillen equivalence $\operad{\C{V}}\rightleftarrows\operad{\C{W}}$ between their operad categories. The conditions so that $F\dashv G$ induces an adjoint pair on operads are too strong. They are not satisfied in many cases of interest, such as the Dold--Kan equivalence $\Mod{\Bbbk}^{\Delta^{\op}}\rightleftarrows\operatorname{Ch}(\Bbbk)_{\geq 0}$
between simplicial modules over a commutative ring $\Bbbk$ and  non-negative chain complexes, see \cite{emmc}. They are neither satisfied in one step of the four-step zig-zag of Quillen equivalences between unbounded chain complexes $\operatorname{Ch}(\Bbbk)$ and modules $\Mod{H\Bbbk}$ over the Eilenberg--MacLane symmetric ring spectrum $H\Bbbk$, see \cite{hzas}.

Schwede and Shipley found in \cite{emmc} sufficient conditions to obtain a Quillen equivalence $\mon(\C{V})\rightleftarrows\mon(\C{W})$ between monoid categories. We construct, under the same hypotheses, a Quillen equivalence between operads.

\begin{thm}\label{oqe}
Let $F\colon\C{V}\rightleftarrows\C{W}\colon G$ be a weak symmetric monoidal Quillen equivalence between cofibrantly generated closed symmetric monoidal model categories with cofibrant tensor units. Suppose $\C{V}$ and $\C{W}$ satisfy the monoid axiom and have sets of generating (trivial) cofibrations  with presentable sources. Then, there is a Quillen equivalence between their model categories of operads
$$\xymatrix{\operad{\C{V}}\ar@<.5ex>[r]^-{F^{\operatorname{oper}}}&\operad{\C{W}}.\ar@<.5ex>[l]^-{G}}$$
Here, operads have the model structure in \cite[Theorem 1.1]{htnso}, with weak equivalences and fibrations defined as in the underlying category.
\end{thm}

Despite operads are monoids in the category $\C{V}^{\mathbb N}$ of sequences of objects in $\C V$ with respect to the composition product, the results of Schwede and Shipley do not apply, since this monoidal category is not left closed. 

If $F\dashv G$ is simply a weak symmetric monoidal Quillen pair,  not necessarily an equivalence, there is still defined a Quillen pair $F^{\operatorname{oper}}\dashv G$ as in Theorem \ref{oqe}, see Proposition \ref{qp}. The functor $F^{\operatorname{oper}}$ is not in general given by $F$ levelwise, but their derived functors coincide.

\begin{prop}\label{casigual}
Let $F\colon\C{V}\rightleftarrows\C{W}\colon G$ be a weak symmetric monoidal Quillen adjunction with $\C{V}$ and $\C{W}$ as in Theorem \ref{oqe}. 
%between cofibrantly generated closed symmetric monoidal model categories with cofibrant tensor units. Suppose $\C{V}$ and $\C{W}$ satisfy the monoid axiom and have sets of generating (trivial) cofibrations  with presentable sources. 
For any operad $\mathcal O$ in $\C V$, there are  natural isomorphisms in $\ho\C W$, $n\geq 0$, 
$$\mathbb{L}F(\mathcal O(n))\cong\mathbb{L}F^{\operatorname{oper}}(\mathcal O)(n).$$
\end{prop}

As a consequence of Theorem \ref{oqe}, we obtain Quillen equivalences between operads over the aforementioned Quillen equivalent symmetric monoidal model categories. 

\begin{cor}
There is a Quillen equivalence $\operad{\Mod{\Bbbk}^{\Delta^{\op}}}\rightleftarrows\operad{\operatorname{Ch}(\Bbbk)_{\geq 0}}$.
\end{cor}

\begin{cor}
The operad categories $\operad{\operatorname{Ch}(\Bbbk)}$ and $\operad{\Mod{H\Bbbk}}$ are Quillen equivalent.
\end{cor}

It is reasonable to wonder whether two operads corresponding under the Quillen equivalence of Theorem \ref{oqe} have Quillen equivalent categories of algebras. The following  result answers this question. 

\begin{thm}\label{qea}
In the situation of Theorem \ref{oqe}:
\begin{enumerate}
\item If $\mathcal O$ is a cofibrant operad in $\C V$, there is a Quillen equivalence between model categories of algebras,
$$\xymatrix{
\algebra{\mathcal O}{\C V}\ar@<.5ex>[r]^-{ F_{\mathcal O}}&\algebra{F^{\operatorname{oper}}(\mathcal O)}{\C W}
\ar@<.5ex>[l]^-{ G}.
}$$
\item If $\mathcal P$ is a fibrant operad in $\C W$ such that $\mathcal P(n)$ and $G(\mathcal P(n))$ are cofibrant for all $n\geq 0$, then there is a Quillen equivalence,
$$\xymatrix{
\algebra{G(\mathcal P)}{\C V}\ar@<.5ex>[r]^-{ F^{\mathcal P}}&\algebra{\mathcal P}{\C W}
\ar@<.5ex>[l]^-{ G}.
}$$
\end{enumerate}
Here, algebras have the model structure in \cite[Theorem 1.2]{htnso}, with weak equivalences and fibrations defined as in the underlying category.
\end{thm}

We will actually prove a more general version of Theorem \ref{qea} where algebras may live in a different category than the operad, see Theorems \ref{qea2}, \ref{qea3}, and \ref{qea4}. 

Again, if $F\dashv G$ is simply a weak symmetric monoidal Quillen pair,  there are still Quillen pairs $F_{\mathcal O}\dashv G$ and $F^{\mathcal P}\dashv G$ as in Theorem \ref{qea}, see Propositions \ref{qp2} and \ref{qp3}. Despite the functors $F_{\mathcal O}$ and $F^{\mathcal P}$ need not coincide with $F$ on underlying objects, their derived functors agree. 

\begin{prop}\label{casigual2}
Let $F\colon\C{V}\rightleftarrows\C{W}\colon G$ be a weak symmetric monoidal Quillen adjunction with $\C{V}$ and $\C{W}$ as in Theorem \ref{oqe}. 
Consider a cofibrant operad $\mathcal O$ in $\C V$ and a fibrant operad $\mathcal P$ in $\C{W}$ such that $\mathcal P(n)$ and $G(\mathcal P(n))$ are cofibrant for all $n\geq 0$. For any $\mathcal O$-algebra $A$ there is a natural isomorphism in $\ho\C W$,
\begin{align*}
\mathbb{L}F(A)&\cong\mathbb{L}F_{\mathcal{O}}(A).
\end{align*} 
If $F\dashv G$ is in addition a Quillen equivalence, then for any $\mathcal P$-algebra $B$ there is a natural isomorphism in $\ho\C W$, 
\begin{align*}
\mathbb{L}F(B)&\cong\mathbb{L}F^{\mathcal{P}}(B).
\end{align*}
\end{prop}

\bigskip

We apply the previous results to show homotopy invariance of the associative operad $\mathtt{Ass}^{\C{V}}$ and the unital associative operad $\mathtt{uAss}^{\C{V}}$. Algebras over these operads are non-unital monoids and unital monoids, respectively. We actually prove homotopy invariance of the natural map $\phi^{\C V}\colon \mathtt{Ass}^{\C{V}}\r \mathtt{uAss}^{\C{V}}$ modelling the forgetful functor from unital to non-unital monoids.

\begin{thm}\label{invariance}
Let $F\colon\C{V}\rightleftarrows\C{W}\colon G$ be a weak symmetric monoidal Quillen adjunction with $\C{V}$ and $\C{W}$ as in Theorem \ref{oqe}. Consider the derived adjoint pair,
$$\xymatrix{\ho\operad{\C{V}}\ar@<.5ex>[r]^-{\mathbb L F^{\operatorname{oper}}}&\ho\operad{\C{W}}.\ar@<.5ex>[l]^-{\mathbb R G}}$$
There are isomorphisms in $\ho\operad{\C{W}}$,
\begin{align*}
\mathbb L F^{\operatorname{oper}}(\mathtt{Ass}^{\C{V}})&\cong \mathtt{Ass}^{\C{W}},&
\mathbb L F^{\operatorname{oper}}(\mathtt{uAss}^{\C{V}})&\cong \mathtt{uAss}^{\C{W}},
%\\
%\mathbb R G^{\operatorname{oper}}\mathtt{Ass}^{\C{W}}&\cong \mathtt{Ass}^{\C{V}},&
%\mathbb R G^{\operatorname{oper}}\mathtt{uAss}^{\C{W}}&\cong \mathtt{uAss}^{\C{V}},
\end{align*}
such that the following square commutes,
$$\xymatrix@C=50pt{
\mathbb L F^{\operatorname{oper}}(\mathtt{Ass}^{\C{V}})\ar[r]^{\mathbb L F^{\operatorname{oper}}(\phi^{\C V})}\ar[d]_{\cong}&
\mathbb L F^{\operatorname{oper}}(\mathtt{uAss}^{\C{V}})\ar[d]^{\cong}\\
\mathtt{Ass}^{\C{W}}\ar[r]^{\phi^{\C W}}&
\mathtt{uAss}^{\C{W}}
}
%\qquad
%\xymatrix@C=30pt{
%\mathbb R G\mathtt{Ass}^{\C{W}}\ar[r]^{\mathbb R G\phi^{\C W}}\ar[d]_{\cong}&
%\mathbb R G\mathtt{uAss}^{\C{W}}\ar[d]^{\cong}\\
%\mathtt{Ass}^{\C{V}}\ar[r]^{\phi^{\C V}}&
%\mathtt{uAss}^{\C{V}}
%}
$$
\end{thm}

For weak symmetric monoidal Quillen equivalences, we deduce the following result by adjunction.

\begin{cor}\label{invariance2}
In the conditions of Theorem \ref{oqe} there are isomorphisms in $\ho\operad{\C{V}}$,
\begin{align*}
%\mathbb L F^{\operatorname{oper}}\mathtt{Ass}^{\C{V}}&\cong \mathtt{Ass}^{\C{W}},&
%\mathbb L F^{\operatorname{oper}}\mathtt{uAss}^{\C{V}}&\cong \mathtt{uAss}^{\C{W}},
%\\
 \mathtt{Ass}^{\C{V}}&\cong\mathbb R G(\mathtt{Ass}^{\C{W}}),&
\mathtt{uAss}^{\C{V}}&\cong \mathbb R G(\mathtt{uAss}^{\C{W}}),
\end{align*}
such that the following square commutes,
$$
%\xymatrix@C=40pt{
%\mathbb L F^{\operatorname{oper}}\mathtt{Ass}^{\C{V}}\ar[r]^{\mathbb L F^{\operatorname{oper}}\phi^{\C V}}\ar[d]_{\cong}&
%\mathbb L F^{\operatorname{oper}}\mathtt{uAss}^{\C{V}}\ar[d]^{\cong}\\
%\mathtt{Ass}^{\C{W}}\ar[r]^{\phi^{\C W}}&
%\mathtt{uAss}^{\C{W}}
%}
%\qquad
\xymatrix@C=40pt{
\mathtt{Ass}^{\C{V}}\ar[r]^{\phi^{\C V}}\ar[d]_{\cong}&
\mathtt{uAss}^{\C{V}}\ar[d]^{\cong}\\
\mathbb R G(\mathtt{Ass}^{\C{W}})\ar[r]^{\mathbb R G(\phi^{\C W})}&
\mathbb R G(\mathtt{uAss}^{\C{W}})
}
$$
\end{cor}

These results have implications for categories of non-unital monoids $\operatorname{Mon}^{nu}(\C V)=\algebra{\mathtt{Ass}^{\C{V}}}{\C V}$ and unital monoids $\mon(\C V)=\algebra{\mathtt{uAss}^{\C{V}}}{\C V}$.

\begin{prop}\label{invariance3}
In the situation of Theorem \ref{oqe}, we have Quillen equivalences between unital and non-unital monoid categories,$$\xymatrix{
\mon(\C{V})\ar@<.5ex>[r]^-{}&\mon(\C{W}),\ar@<.5ex>[l]^-{G}
&
\operatorname{Mon}^{nu}(\C V)\ar@<.5ex>[r]^-{}&\operatorname{Mon}^{nu}(\C W)\ar@<.5ex>[l]^-{G}.}$$
\end{prop}

The  Quillen equivalence on the left was obtained in \cite{emmc}. The  Quillen equivalence on the right is new but could also have been obtained by the same methods. 

If we consider algebras over these operads in the category $\operatorname{Graph}_{S}(\C V)$ of $\C V$-graphs with a fixed object set $S$, see \cite[\S10]{htnso}, we obtain a similar result for non-unital $\C V$-enriched categories $\operatorname{Cat}^{nu}_{S}(\C V)=\algebra{\mathtt{Ass}^{\C{V}}}{\operatorname{Graph}_{S}(\C V)}$ and unital $\C V$-enriched categories $\operatorname{Cat}_{S}(\C V)=\algebra{\mathtt{uAss}^{\C{V}}}{\operatorname{Graph}_{S}(\C V)}$ with fixed object set $S$.

\begin{prop}\label{invariance4}
In the situation of Theorem \ref{oqe}, we have Quillen equivalences between categories of unital and non-unital $\C V$-enriched categories with fixed object set $S$,
$$\xymatrix{
\operatorname{Cat}_{S}(\C{V})\ar@<.5ex>[r]^-{}&\operatorname{Cat}_{S}(\C{W}),\ar@<.5ex>[l]^-{G}
&
\operatorname{Cat}^{nu}_{S}(\C V)\ar@<.5ex>[r]^-{}&\operatorname{Cat}^{nu}_{S}(\C W).\ar@<.5ex>[l]^-{G}}$$
\end{prop}

\bigskip

Recall that a model category is \emph{left proper} if the push-out of a weak equivalence $f$ along a cofibration $g$ is always a weak equivalence $f'$,
$$\xymatrix{X\ar@{ >->}[r]^-{g}\ar[d]_-{f}^{\sim}\ar@{}[rd]|{\text{push}}&Z\ar[d]^-{f'}_{\sim}\\
Y\ar@{ >->}[r]_-{g'}&Y\cup_{X}Z}$$
In general, this property is only satisfied if $f$ is a trivial cofibration or if $X$ and $Y$ are cofibrant. In particular, any model category whose objects are all cofibrant is left proper. The model category of operads $\operad{\C{V}}$ is not left proper, even if $\C{V}$ is. Nevertheless, we here show the following result along this line.

\begin{thm}\label{lp}
Let $\C{V}$ be a cofibrantly generated closed symmetric monoidal model category satisfying the monoid axiom with sets of generating (trivial) cofibrations  with presentable sources. Consider a push-out diagram in $\operad{\C{V}}$ as follows,
$$\xymatrix{\mathcal{O}\ar@{ >->}[r]^-{\psi}\ar[d]_-{\varphi}^{\sim}\ar@{}[rd]|{\text{push}}&\mathcal{Q}\ar[d]^-{\varphi'}\\
\mathcal{P}\ar@{ >->}[r]_-{\psi'}&\mathcal{P}\cup_{\mathcal{O}}\mathcal{Q}}$$
If $\mathcal{O}(n)$ and $\mathcal{P}(n)$ are cofibrant in $\C{V}$ for all $n\geq 0$, then $\varphi'$ is a weak equivalence. 
\end{thm}

If $\C{V}$ has a cofibrant tensor unit, the last condition is satified by cofibrant operads,  but also by many non-cofibrant operads of interest, such as the associative operad $\mathtt{Ass}^{\C{V}}$ and the unital associative operad $\mathtt{uAss}^{\C{V}}$. We take advantage of this result in \cite{udga}, where we show that the natural map $\phi^{\C V}\colon \mathtt{Ass}^{\C{V}}\r \mathtt{uAss}^{\C{V}}$ considered above is a homotopy epimorphism for a wide class of base categories $\C{V}$.

We would like to stress that Theorem \ref{lp} does not require left properness for $\C V$. This theorem yields left properness for operads in the following special case.

\begin{cor}\label{lp2}
Let $\C{V}$ be as in Theorem \ref{lp}.
Assume all objects in $\C V$ are cofibrant. Then $\operad{\C V}$ is left proper.
\end{cor}

This result applies, for instance, to the category $\operatorname{Set}^{\Delta^{\op}}$ of simplicial sets  and to the category $\operatorname{Ch}(\Bbbk)$ of unbounded chain complexes over a field $\Bbbk$.

We have similar results for algebras.

\begin{thm}\label{lp3}
Let $\C{V}$ be as in Theorem \ref{lp}.
Consider an operad $\mathcal O$ in $\C V$ such that $\mathcal O(n)$ is cofibrant for all $n\geq 0$, and a push-out diagram in the category of $\mathcal O$-algebras as follows,
$$\xymatrix{A\ar@{ >->}[r]^-{\psi}\ar[d]_-{\varphi}^{\sim}\ar@{}[rd]|{\text{push}}&C\ar[d]^-{\varphi'}\\
B\ar@{ >->}[r]_-{\psi'}&B\cup_{A}C}$$
If the underlying objects of $A$ and $B$ are cofibrant in $\C V$ then $\varphi'$ is a weak equivalence. 
\end{thm}

\begin{cor}\label{lp3.5}
Let $\C{V}$ be as in Theorem \ref{lp}.  Assume all objects in $\C V$ are cofibrant. Then, $\algebra{\mathcal O}{\C V}$ is left proper for any operad $\mathcal O$ in $\C V$.
\end{cor}

These results will also be proved for algebras living in a different category than the operad, see Theorem \ref{lp5} and Corollary \ref{lp6}.

\bigskip

The reader must have noticed the cofibrancy hypotheses on tensor units in many results above. 
In the appendices, we show how to get rid of them, still imposing some (weaker) hypotheses. The tensor unit is not cofibrant in some examples of interest, such as symmetric spectra and other diagram spectra with the positive stable model structure \cite{mcds}. Positive stable model structures are very important, actually unavoidable in brave new algebraic geometry, where they must be used in order to have transferred model structures on commutative monoids \cite{hagII}. In \cite{manso} we construct moduli spaces of algebras over an operad in brave new algebraic geometry. This is why we need to avoid cofibrancy hypotheses on tensor units.

In Appendix \ref{pco} we isolate a class of objects, that we call \emph{pseudo-cofibrant},  which share many properties with cofibrant objects. The tensor unit is always pseudo-cofibrant, and the components $\mathcal{O}(n)$, $n\geq 0$, of a cofibrant operad $\mathcal{O}$ are also pseudo-cofibrant. 

Pseudo-cofibrant objects share even more properties with cofibrant objects when the \emph{strong unit axiom} is satisfied. This axiom holds, for instance, whenever cofibrant objects are flat, i.e.~if tensoring with a cofibrant object preserves weak equivalences. This is a rather common property, satisfied by symmetric spectra with the positive model structure. 

Appendix \ref{ico} analyzes a relevant subclass of pseudo-cofibrant objects: the \emph{$\unit$-cofibrant} objects, i.e.~those objects $X$ admitting a cofibration $\unit\into X$ from the tensor unit. 

Weak monoidal Quillen adjunctions must satisfy two axioms in order to extend the previous results: the \emph{pseudo-cofibrant} and the \emph{$\unit$-cofibrant axioms}, introduced in Appendix \ref{ico}. These axioms mitigate the fact that left Quillen functors need not preserve pseudo-cofibrant objects and weak equivalences between them.

The three new axioms are obviously satisfied when tensor units are cofibrant. They are also satisfied in many other examples related to spectra where  tensor units are not cofibrant. The statements of our main results with these weaker hypotheses are in Appendices \ref{geno}, \ref{gena},  and \ref{genainf}.

\bigskip

The paper is structured as follows. We first prove change of base category and left properness results for operads and then for algebras. The homotopy invariance of the (unital) associative operad comes later, just before the appendices. We intercalate three sections with background on monoidal model categories, operads, and algebras over operads.

We assume the reader familiarity with category theory and abstract homotopy theory. Some standard references are \cite{cwm2,hmc, hirschhorn}. For monoidal categories, functors, and adjunctions, we refer to \cite[Chapter 3]{mfsha}.

\subsection*{Acknowledgements}

The author is grateful to David White for Lemma \ref{davidwhite}, suggested as an answer to an author's question in MathOverflow. %He is also grateful to Michael Batanin and Clemens Berger for conversations on the results of this paper.

The author was partially supported
by the Andalusian Ministry of Economy, Innovation and Science under the grant FQM-5713, by the Spanish Ministry of Education and
Science under the MEC-FEDER grant  MTM2010-15831, and by the Government of Catalonia under the grant SGR-119-2009.

 \section{Monoidal model categories}

We here recall  the compatibility conditions between monoidal and model structures introduced in \cite{hmc, ammmc,emmc}. See also \cite{htnso} for the non-symmetric version of the monoid axiom.

The tensor product and tensor unit of a monoidal category $\C C$ will be denoted by $\otimes$ and $\unit$, respectively. We will sometimes write $\otimes_{\C C}$ and $\unit_{\C C}$ if we wish to distinguish between different monoidal categories. Initial objects will be denoted by $\varnothing$.

\begin{defn}
Given a category $\C{C}$,  the \emph{category of morphisms} $\mor{\C{C}}$ is the category of functors $\dos\r\C{C}$, where $\dos$ is the category with two objects, $0$ and $1$, and only one non-identity morphism $0\r 1$, i.e.~it is the poset $\{0<1\}$. A morphism $f\colon U\r V$ in $\C C$ is identified with the functor $f\colon\dos\r\C{C}$ defined by $f(0)=U$, $f(1)=V$ and $f(0\r 1)=f$.

If $\C{C}$ is a cocomplete biclosed (symmetric) monoidal category, then $\mor{\C{C}}$ is biclosed (symmetric)  monoidal with respect to the {push-out product} of morphisms. Given morphisms $f\colon U\r V$ and $g\colon X\r Y$ in $\C{C}$,  the \emph{push-out product} $f\odot g$ is defined by  the following diagram 
\begin{equation}\label{pushoutproduct}
\xy/r2.3pt/:
(0,0)*+{U\otimes X}="a",
(35,0)*+{V\otimes X}="b",
(0,-18.2)*+{U\otimes Y}="c",
(35,-20)*+{U\otimes Y\!\bigcup\limits_{U\otimes X}\! V\otimes X}="d",
(70,-35)*+{V\otimes Y}="e"
\ar"a";"b"^{f\otimes  {X}}
\ar@{}"a";"d"|-{\text{push}}\ar"a";"c"_{ {U}\otimes g}
\ar"b";"d"^-{\bar{g}}
\ar"c";(17,-18.2)^-{\bar{f}}
\ar@/^15pt/"b";"e"^{ {V}\otimes g}
\ar@/_15pt/"c";"e"_{f\otimes {Y}}
\ar(40,-23);"e"^-{f\odot g}
\endxy
\end{equation}
Notice that $f\odot(\varnothing\r X)=f\otimes X$ and $(\varnothing\r X)\odot f=X\otimes f$. The unit object in $\mor{\C{C}}$ is $\varnothing\r\unit$. 
\end{defn}

\begin{rem}\label{iteracion}
Given morphisms $f_i\colon U_i\rightarrow V_i$ in $\C C$, $1\leq i\leq n$, the iterated push-out product $f_1\odot\cdots\odot f_n$ can be constructed as follows. Consider the diagram
$$f_{1}\otimes\cdots\otimes f_{n}\colon\dos^{n}\To\C{C}.$$
Then $f_1\odot\cdots\odot f_n$ is the morphism
$$\colim_{\dos^{n}\setminus\{(1,\st{n}\dots,1)\}} f_{1}\otimes\cdots\otimes f_{n}\To \colim_{\dos^{n}}f_{1}\otimes\cdots\otimes f_{n}=V_{1}\otimes\cdots\otimes V_{n}$$
induced by the inclusion of the full subcategory $\dos^{n}\setminus\{(1,\st{n}\dots,1)\}\subset\dos^{n}$ obtained by removing the final object of $\dos^{n}$, which is the  $n$-dimensional cube category. %, i.e.~the matching category of $\dos^{n}$ at $(1,\st{n}\dots,1)$. 
\end{rem}

\begin{defn}
A \emph{monoidal model category} is a biclosed   monoidal category $\C{C}$ endowed with a model structure such that the following axioms hold:
\begin{enumerate}
\item \emph{Push-out product axiom}: If $f$ and $g$ are cofibrations then $f\odot g$ is also a cofibration. Moreover, if in addition $f$ or $g$ is a trivial cofibration then $f\odot g$ is a trivial cofibration.

\item \emph{Unit axiom}: If $X$ is a cofibrant object and $q\colon \tilde\unit\st{\sim}\onto\unit$ is a cofibrant resolution of the tensor unit, then $X\otimes q$ and $q\otimes X$ are weak equivalences.
\end{enumerate}
These two axioms imply that the homotopy category $\ho\C C$ is biclosed   monoidal, see \cite{hmc}.

In order to have transferred model structures on monoids, operads, algebras over operads, etc.~we need to impose the following extra axiom, see \cite{ammmc,htnso}:
\begin{enumerate}\setcounter{enumi}{2}
\item \emph{Monoid axiom}: Relative $K'$-cell complexes are weak equivalences, where
\begin{align*}
K'={}&\{f_{1}\odot\cdots \odot f_{n}\;;\;n\geq 1,\, \varnothing\neq S\subset\{1,\dots,n\}, \,
\\&
\qquad\qquad \qquad\quad f_{i}\text{ is a trivial cofibration if }i\in S,\\
& \qquad\qquad \qquad\quad f_{i}\colon \varnothing\r X_{i} \text{ for some object }X_{i}\text{ in }\C{C}\text{ if }i\notin S\}.
\end{align*}
\end{enumerate}
This axiom is usually not incorporated in the definition of monoidal model category but we do include it, since we will always need it. We will also assume that all monoidal model categories are cofibrantly generated. Moreover, we will suppose that there is a set of generating cofibrations $I$ and a set of generating trivial cofibrations $J$ with presentable sources.

A \emph{symmetric monoidal model category} $\C{V}$ is a monoidal model category as above whose underlying monoidal category is symmetric. In this case, (3) can be replaced with
\begin{itemize}
\item[$(3')$] \emph{Monoid axiom}: Relative $K$-cell complexes are weak equivalences, where
$$K=\{f\otimes X\,;\,f\text{ is a trivial cofibration and }X\text{ is an object in }\C{V}\}.$$
\end{itemize}

\end{defn}

\begin{rem}
If the unit axiom holds for a certain cofibrant resolution of $\unit$ then  it holds for any cofibrant resolution of $\unit$. It is obviously satisfied if $\unit$ is cofibrant. 

In many examples, tensoring with a cofibrant object preserves weak equivalences. This is a usual condition to ensure that a weak equivalence of monoids induces a Quillen equivalence between their module categories, see \cite{ammmc}. This condition also implies the unit axiom.

The push-out product axiom has many immediate consequences that the reader may work out, e.g.~cofibrant objects are closed under tensor products, cofibrations between cofibrant objects are closed under push-out products, weak equivalences between cofibrant objects are closed under tensor products, etc.
\end{rem}

The following result is another  straightforward consequence of the push-out product axiom.

\begin{lem}\label{reedy}
Given cofibrations between cofibrant objects  $f_{1},\dots, f_{n}$ in a monoidal model category $\C{C}$, the diagram
$$f_{1}\otimes \cdots\otimes f_{n}\colon\dos^{n}\To\C{C}$$
is Reedy cofibrant in $\C{C}^{\dos^{n}}$.
\end{lem}

\begin{proof}
Denote $f_{i}\colon U_{i}\into V_{i}$, $1\leq i\leq n$. 
The diagram $f_{1}\otimes \cdots\otimes f_{n}$ is Reedy cofibrant if and only if for each subset $S\subset\{1,\dots,n\}$, the morphism
$\bigodot_{s\in S}f_{s}\otimes\bigotimes_{t\notin S}U_{t}$ is a cofibration. This condition for $S=\varnothing$ says that $U_{1}\otimes\cdots\otimes U_{n}$ must be cofibrant. The tensor and push-out product factors must be reordered according to the subscript if $\C{C}$ is non-symmetric. This property follows from the push-out product axiom.
\end{proof}

\begin{defn}\label{weak}
A Quillen pair between monoidal model categories $F\colon\C{C}\rightleftarrows\C{D}\colon G$ is a \emph{weak monoidal Quillen adjunction} in the sense of \cite{emmc} if $F\dashv G$ is a colax-lax monoidal adjunction and the following two properties hold:
\begin{enumerate}
\item If $X$ and $Y$ are cofibrant objects in $\C{C}$, the comultiplication of the colax monoidal functor $F$,
$$F(X\otimes_{\C{C}}Y)\st{\sim}\To F(X)\otimes_{\C{D}}F(Y),$$
is a weak equivalence. 
\item If $q\colon \tilde\unit_{\C{C}}\st{\sim}\onto\unit_{\C{C}}$ is 
a cofibrant resolution then the composite
$$\xymatrix{F(\tilde\unit_{\C{C}})\ar[r]^-{F(q)}& F(\unit_{\C{C}})\ar[r]^-{\text{counit }}& \unit_{\C{D}}}$$
is a weak equivalence. 
\end{enumerate}

A \emph{weak monoidal Quillen equivalence} is a weak monoidal Quillen adjunction which is a Quillen equivalence.

A \emph{weak symmetric monoidal Quillen adjunction} is a weak monoidal Quillen adjunction such that the colax (resp.~lax) monoidal functor $F$ (resp.~$G$) is symmetric. We similarly define a \emph{weak symmetric monoidal Quillen equivalence}.

If $F$ is strong, we drop the word `weak' from the previous terminology, and speak of a \emph{monoidal Quillen adjunction}, etc. These monoidal Quillen pairs were considered in \cite{hmc}.

% The left (resp.~right) adjoint in a weak monoidal Quillen adjunction is called a \emph{left} (resp.~\emph{right}) \emph{weak monoidal Quillen functor}. This terminilogy will be accordingly modified if the weak monoidal Quillen adjunction is symmetric, an equivalence, etc.
\end{defn}

Property (1) is suitable for iteration.

\begin{lem}\label{iterated}
Given  a {weak  monoidal Quillen adjunction}  $F\colon\C{C}\rightleftarrows\C{D}\colon G$ and cofibrant objects $X_{1},\dots, X_{n}$  in $\C{C}$, $n\geq 1$,  the iterated comultiplication
$$F\left(\bigotimes_{i=1}^{n}X_{i}\right)\To \bigotimes_{i=1}^{n}F(X_{i})$$
is a weak equivalence in $\C{D}$.
\end{lem}

\begin{proof}
The case $n=2$ is (1) above. Suppose the result is true for the tensor product of $n-1$ cofibrant objects, $n>2$. The morphism in the statement decomposes as
$$F\left(\bigotimes_{i=1}^{n-1}X_{i}\otimes X_{n}\right)\To F\left(\bigotimes_{i=1}^{n-1}X_{i}\right)\otimes F(X_{n})\To \bigotimes_{i=1}^{n-1}F(X_{i})\otimes F(X_{n}).$$
The first arrow is a weak equivalence, it is the case $n=2$. The second arrow is a weak equivalence by the monoid axiom, since it is the tensor product of a weak equivalence bewteen cofibrant objects (the case $n-1$) with another object, $F(X_{n})$. 
\end{proof}

We finally recall the notion of model algebra over a symmetric monoidal model category. 

\begin{defn}\label{valgebra}
If $\C V$ is a symmetric monoidal category, a \emph{$\C V$-algebra} $\C C$ is a monoidal category equipped with a strong monoidal functor $z\colon\C V\r\C C$ and a  natural isomorphisms,
\begin{align*}
\zeta(X,Y)\colon z(X)\otimes_{\C C}Y&\cong Y\otimes_{\C C}z(X),
\end{align*}
satisfying certain coherence conditions, see \cite[\S7]{htnso}. This is the same as a strong braided monoidal functor $\C V\r Z(\C C)$ to the braided center of $\C C$ in the sense of \cite{tybotc}.

Let $\C{V}$ be a symmetric monoidal model category. A \emph{model $\C{V}$-algebra} $\C{C}$ is a monoidal model category which is a $\C{V}$-algebra in such a way that $z\colon\C{V}\r\C{C}$ is a left Quillen functor. These model algebras are termed central in \cite{hmc}. We will also assume that all model $\C{V}$-algebras are cofibrantly generated. Moreover, we will suppose that there is a set of generating cofibrations $I'$ and a set of generating trivial cofibrations $J'$ in $\C C$ with presentable sources.
\end{defn}

\section{Operads}\label{operads}

In this section we recall some facts about (non-symmetric) operads that will be used throughout the paper. We refer the reader to \cite{htnso} for further details.

\begin{defn}
Let $\C{V}$ be a closed symmetric monoidal category with coproducts. 
 The category~$\C{V}^{\mathbb{N}}$ of  \emph{sequences} of objects $V=\{V(n)\}_{n\geq0}$ in~$\C{V}$ 
is the product of countably many copies of $\C{V}$, $\C{V}^{\mathbb{N}}=\prod_{n\geq 0}\C{V}$. It
 has a right-closed non-symmetric monoidal structure given by the \emph{composition product} $\circ$, compare \cite[Definition 1.2]{hcctaimc},
\begin{align*}
(U\circ V)(n)&=\coprod_{m\geq 0}\coprod_{\sum\limits_{i=1}^{m}p_{i}=n}U(m)\otimes V(p_{1})\otimes\cdots\otimes V(p_{m}).
\end{align*}
The unit object $\unit_{\circ}$ is,
$$\unit_{\circ}(n)=\left\{
\begin{array}{ll}
\unit,&n=1;\\
\varnothing,&n\neq 1.
\end{array}
\right.$$
We say that $V(n)$ is the \emph{arity} $n$ component of an object $V$ in $\C{V}^{\mathbb{N}}$.

An \emph{operad} $\mathcal O$ in  $\C{V}$ is a monoid in $\C{V}^{\mathbb{N}}$ for the composition product. The category of operads in $\C{V}$ will be denoted by $\operad{\C{V}}$.
\end{defn}

\begin{rem}
The previous condensed definition of  operad $\mathcal{O}$ can be unraveled by noticing that the 
monoid structure $\mathcal{O}\circ \mathcal{O}\r \mathcal{O}$, $\unit_\circ\r\mathcal{O}$, consists of a series of \emph{multiplication morphisms}, $n\geq 1$, $p_{i}\geq 0$, 
$$\mu_{n;p_1,\dots,p_n}\colon \mathcal{O}(n)\otimes \mathcal{O}(p_{1})\otimes\cdots\otimes \mathcal{O}(p_{n})\To \mathcal{O}(p_1+\cdots +p_n),$$
and a \emph{unit},
$$\id{\mathcal{O}}\colon\unit\r \mathcal{O}(1).$$ The associativity and unit laws can be translated into commutative diagrams built up from these morphisms.

There is an even simpler but equivalent characterization of operads in terms of  $\id{\mathcal{O}}$ and  \emph{composition laws}, $1\leq i\leq p$, $q\geq 0$,
$$\circ_i\colon \mathcal{O}(p)\otimes \mathcal{O}(q)\To \mathcal{O}(p+q-1).$$

A \emph{morphism} of operads $f\colon \mathcal{O}\r  \mathcal{P}$ is therefore a sequence of maps $f(n)\colon \mathcal{O}(n)\r  \mathcal{P}(n)$ in $\C{V}$ compatible with the morphisms $\mu_{n,p_{1},\dots,p_{n}}$ and with the units, or with the composition laws $\circ_{i}$ and with the units.
\end{rem}

The forgetful functor from operads to sequences has a left adjoint, the \emph{free operad} functor
$$\xymatrix{\C{V}^{\mathbb{N}}\ar@<.5ex>[r]^-{\mathcal{F}}&\operad{\C{V}}.\ar@<.5ex>[l]^-{\text{forget}}}$$
See \cite[\S5]{htnso} for an explicit construction of free operads.

\begin{rem}\label{pushfree}
If  $\C{V}$ is complete and cocomplete then so is $\operad{\C{V}}$.  Limits and filtered colimits in $\operad{\C{V}}$ are easy, they can be computed on underlying sequences, and limits and colimits of sequences are computed in $\C{V}$ levelwise.  

Push-outs in $\operad{\C{V}}$ are very complicated in general, but push-outs along free maps were carefully analyzed in \cite[\S 5]{htnso}, %. A diagram of sequences
%$$\mathcal{O}\st{\bar g}\longleftarrow U\st{f}\To V$$
%where $\mathcal{O}$ is an operad is essentially the same as a diagram in $\operad{\C{V}}$
%$$\mathcal{O}\st{g}\longleftarrow \mathcal{F}(U)\st{f}\To \mathcal{F}(V).$$
%Consider the push
\begin{equation}\label{pusho}
\xymatrix{\mathcal{F} (U)\ar@{}[rd]|{\text{push}}\ar[r]^-{\mathcal{F} (f)}\ar[d]_-{g}&\mathcal{F} (V)\ar[d]^-{g'}\\
\mathcal{O}\ar[r]_-{f'}&\mathcal{P}}
\end{equation}
The morphism of sequences underlying $f'$ is a transfinite (actually countable) composition of morphisms
$$\mathcal{O}=P_0\st{\varphi_1}\To P_1\r\cdots\r P_{t-1}\st{\varphi_t}\To P_t\r \cdots.$$
The morphism $\varphi_t(n)\colon P_{t-1}(n)\r P_t(n)$
is a push-out along
\begin{equation}\label{monstruo}
\coprod_{T}\bigodot_{v\in I^{e}(T)}f(\val{T}{v})\;\;\otimes\bigotimes_{w\in I^{o}(T)}\mathcal{O}(\val{T}{w}).
\end{equation}
Here $T$ runs over the isomorphism classes of planted planar trees with $n$ leaves concentrated in even levels and $t$ inner vertices in even levels, $I^e(T)$ is the set of inner vertices in even levels, $I^o(T)$ is the set of inner vertices in odd levels, and $\val{T}{v}+1$ is the number of edges  adjacent to a given a vertex $v$. See \cite[\S 3]{htnso} for the combinatorics of trees needed to deal with operads. The map from the source of \eqref{monstruo} to $P_{t-1}(n)$ is tedious to describe, we will not use its explicit definition in this paper, we refer the reader to \cite[Lemma 5.1]{htnso}.
\end{rem}

\begin{rem}
Let $\C{V}$ be a symmetric monoidal model category. The category of sequences $\C{V}^{\mathbb{N}}$ is a model category with the product model structure. A morphism $f\colon U\r V$ in  $\C{V}^{\mathbb{N}}$  is a fibration, cofibration, or weak equivalence, if $f(n)\colon U(n)\r V(n)$ is so for all $n\geq 0$. However, $\C{V}^{\mathbb{N}}$ is not a monoidal model category since it is not left closed, so the results of \cite{ammmc,emmc} do not apply to operads.

%Suppose $\C{V}$ is cofibrantly generated with sets of generating (trivial) cofibrations $I$ and $J$.

The model category $\C{V}^{\mathbb{N}}$ is cofibrantly generated with sets of generating (trivial) cofibrations $I_{\mathbb{N}}$ and $J_{\mathbb{N}}$. Here, if $S$ is a set of morphisms in $\C{V}$, we denote $S_{\mathbb{N}}$ the set of morphisms in $\C{V}^{\mathbb{N}}$ defined as follows: $f\in S_{\mathbb{N}}$ if there exists $m\geq 0$ such that $f(m)\in S$ and $f(n)\colon \varnothing\r \varnothing$ is the identity on the initial object for all $n\neq m$.

We showed in  \cite[Theorem 1.1]{htnso} that the category of operads $\operad{\C{V}}$ has a  cofibrantly generated transferred model structure along the free operad adjunction, i.e.~weak equivalences and fibrations are defined as in the underlying category of sequences. Sets of generating (trivial) cofibrations are 
 $\mathcal{F}(I_{\mathbb{N}})$ and $\mathcal{F}(J_{\mathbb{N}})$. This is the model structure on operads that we will always consider in this paper.
\end{rem}

The following lemma is a straightforward consequence of the push-out product axiom.

\begin{lem}\label{sequicof}
If $\C V$ is a symmetric monoidal model category, $f$ is a cofibration in $\C{V}^{\mathbb{N}}$, and $\mathcal{O}$ is an operad in $\C{V}$ such that $\mathcal{O}(n)$ is cofibrant for all $n\geq 0$, the morphism \eqref{monstruo} is a cofibration between cofibrant objects in $\C{V}$.
\end{lem}

We derive from this lemma some useful properties of cofibrations and cofibrant objects in $\operad{\C{V}}$.

\begin{cor}\label{yi}
Consider a push-out diagram \eqref{pusho} in $\operad{\C{V}}$. If   $\C V$ is a symmetric monoidal model category, $f$ is a cofibration in $\C{V}^{\mathbb{N}}$, and $\mathcal{O}(n)$ is cofibrant for all $n\geq 0$, then $f'(n)\colon \mathcal{O}(n)\r \mathcal{P}(n)$ is a cofibration for all $n\geq 0$, in particular $\mathcal{P}(n)$ is cofibrant.
\end{cor}

\begin{proof}
By the previous lemma, $f'(n)$ is a transfinite (countable) composition of cofibrations, hence a cofibration itself.
\end{proof}

\begin{cor}\label{kisin}
Let  $\C V$ be a symmetric monoidal model category. 
If $\varphi\colon\mathcal{O}\into\mathcal{P}$ is a cofibration in $\operad{\C V}$ and $\mathcal{O}(n)$ is cofibrant in $\C{V}$ for all $n\geq 0$, then $\varphi(n)\colon \mathcal{O}(n)\r \mathcal{P}(n)$ is a cofibration for all $n\geq 0$, in particular $\mathcal{P}(n)$ is cofibrant.
\end{cor}

\begin{proof}
If $\varphi$ is a relative $\mathcal{F}(I_{\mathbb{N}})$-cell complex, the previous corollary shows that $\varphi(n)$ is a transfinite composition of cofibrations, therefore a cofibration, $n\geq 0$.  Cofibrations in $\operad{\C V}$ are retracts of $\mathcal{F}(I_{\mathbb{N}})$-cell complexes, hence the result follows since retracts of cofibrations in $\C{V}$ are cofibrations.
\end{proof}

\begin{cor}\label{tech1.5bis}
Suppose the tensor unit of the symmetric monoidal model category $\C{V}$ is cofibrant. If $\mathcal{O}$ is cofibrant in $\operad{\C V}$ then $\mathcal{O}(n)$ is cofibrant in $\C{V}$ for all $n\geq 0$.
\end{cor}

\begin{proof}
If $\unit$ is cofibrant then the initial operad $\unit_{\circ}$ is levelwise cofibrant. Hence this corollary follows from the previous one.
\end{proof}

\section{Change of base category for operads}

In this section we prove Theorem \ref{oqe} and Proposition \ref{casigual}. The key step is Proposition \ref{paco}. This result will also be used in the proof of Theorem \ref{qea} and in applications to (unital) $A$-infinity algebras.

A  lax symmetric monoidal functor $\Phi\colon\C{V}\r\C{W}$ between symmetric monoidal categories with coproducts induces a functor between categories of operads,
$$\Phi\colon\operad{\C{V}}\To \operad{\C{W}}.$$
Indeed, $\Phi$ applied levelwise induces a functor between sequences $\Phi\colon \C{V}^{\mathbb{N}}\r \C{W}^{\mathbb{N}}$ which is lax monoidal
for the composition product. The multiplication is given by
$$
\xy
(0,0)*{\displaystyle(\Phi(U)\circ_{\C{W}} \Phi(V))(n)=\coprod_{m\geq 0}\coprod_{\sum\limits_{i=1}^{m}p_{i}=n}\Phi(U(m))\otimes_{\C{W}} \Phi(V(p_{1}))\otimes_{\C{W}}\cdots\otimes_{\C{W}} \Phi(V(p_{m}))},
(11.3,-20)*{\displaystyle \coprod_{m\geq 0}\coprod_{\sum\limits_{i=1}^{m}p_{i}=n}\Phi(U(m)\otimes_{\C{V}} V(p_{1})\otimes_{\C{V}}\cdots\otimes_{\C{V}} V(p_{m}))},
(-1.7,-40)*{\displaystyle\Phi(U\circ_{\C{V}} V)(n)=\Phi(\coprod_{m\geq 0}\coprod_{\sum\limits_{i=1}^{m}p_{i}=n}U(m)\otimes_{\C{V}} V(p_{1})\otimes_{\C{V}}\cdots\otimes_{\C{V}} V(p_{m}))}
\ar(17,0);(17,-13)^-{\coprod_{m\geq 0}\coprod_{\sum\limits_{i=1}^{m}p_{i}=n}\text{mult.}}
\ar(17,-20);(17,-33)%^-{(F(\text{incl. of factor of }\coprod))_{m\geq0,\sum\limits_{i=1}^{m}p_{i}=n}}
\ar(-45,0);(-45,-33)_{\text{mult.}}
\endxy$$
Here, the bottom right vertical morphism is the canonical map $\coprod\Phi\r\Phi(\coprod)$. 
Moreover, the unit morphism $\unit_{\circ_{\C{W}}}\r \Phi(\unit_{\circ_{\C{V}}})$ 
is the unit morphism $\unit_{\C{W}}\r \Phi(\unit_{\C{V}})$ in arity $1$ and 
the trivial morphism  $\varnothing\r \Phi(\varnothing)$ in arities $n\neq 1$. 

The following result shows that colax-lax symmetric monoidal adjoint pairs which are Quillen pairs give rise to Quillen adjunctions between categories of operads. No further conditions are needed. Problems come when we want to obtain a Quillen equivalence.

\begin{prop}\label{qp}
Let $F\colon\C{V}\rightleftarrows\C{W}\colon G$ be a colax-lax symmetric monoidal pair which is also a Quillen adjunction. Then $F\dashv G$ %where $\unit_{\C{V}}$ and $\unit_{\C{W}}$ are cofibrant
gives rise to a Quillen  pair 
between model categories of operads,
$$\xymatrix{\operad{\C{V}}\ar@<.5ex>[r]^-{F^{\operatorname{oper}}}&\operad{\C{W}}.\ar@<.5ex>[l]^-{G}}$$
%Here $L^{\operatorname{oper}}$ is the left adjoint of the functor induced on operads by the lax symmetric monoidal functor $R$. 
If $F$ is strong  then $F^{\operatorname{oper}}=F$. 
\end{prop}

\begin{proof}
Since the lax symmetric monoidal functor $G$ takes (trivial) fibrations in $\C{W}$ to (trivial) fibrations in $\C{V}$, then so does the induced functor between categories of operads. The left adjoint $F^{\operatorname{oper}}$ exists by abstract reasons, see \cite[Theorem 4.5.6]{borceux2}, therefore $F^{\operatorname{oper}}\dashv G$ is a Quillen pair. In general $F^{\operatorname{oper}}$ is not defined as $F$ on underlying sequences. Nevertheless, if $F$  is strong, then the functors induced by $F$ and $G$ on operads are adjoint. Hence $F^{\operatorname{oper}}=F$ in this special case. 
\end{proof}

The functor $F^{\operatorname{oper}}$ may not coincide with  $F$  on underlying sequences, but how different are they? The following proposition answers this question. On cofibrant operads, they coincide up to homotopy.

In the conditions of the previous proposition, consider the following diagram of adjoint pairs
$$\xymatrix{
\C{V}^{\mathbb{N}}\ar@<.5ex>[r]^-{F}\ar@<.5ex>[d]^{\mathcal{F}_{\C V}}\ar@{<-}@<-.5ex>[d]_{\text{forget}}
&\C{W}^{\mathbb{N}}\ar@<.5ex>[l]^-{G}\ar@<.5ex>[d]^{\mathcal{F}_{\C W}}\ar@{<-}@<-.5ex>[d]_{\text{forget}}\\
\operad{\C{V}}\ar@<.5ex>[r]^-{F^{\operatorname{oper}}}&\operad{\C{W}}\ar@<.5ex>[l]^-{G}
}$$
The subdiagram of right adjoints commutes, hence left adjoints commute up to natural isomorphism. Moreover, given an operad $\mathcal{O}$ in $\C{V}$, there is a natural morphism in $\C{W}^{\mathbb{N}}$
\begin{equation}\label{chio}
\chi_{\mathcal O}\colon F(\mathcal O)\To F^{\operatorname{oper}}(\mathcal O)
\end{equation}
whose adjoint $\mathcal O\r GF^{\operatorname{oper}}(\mathcal O)$ along $F\dashv G$ is the morphism of sequences underlying the unit of $F^{\operatorname{oper}}\dashv G$.

\begin{prop}\label{paco}
Let $F\colon\C{V}\rightleftarrows\C{W}\colon G$ be a weak symmetric monoidal Quillen adjunction. Suppose that the tensor units of $\C{V}$ and $\C{W}$ are cofibrant. If $\mathcal{O}$ is cofibrant in $\operad{\C V}$ then $\chi_ {\mathcal{O}}$ is a weak equivalence in $\C{W}^{\mathbb{N}}$. % and $\C{W}$ is left proper. 
\end{prop}

Notice that Proposition \ref{casigual} is a corollary of this result. In the proof of Proposition \ref{paco}, we need the following technical result.

\begin{lem}\label{previo}
Let $F\colon\C{C}\rightleftarrows\C{D}\colon G$ be a weak monoidal Quillen adjunction. Consider cofibrations with cofibrant source and target $f_{1},\dots, f_{n}$ and cofibrant objects $X_{1},\dots, X_{m}$ in $\C{C}$. The morphism in $\mor{\C{D}}$
$$F\left(\bigodot_{i=1}^{n}f_{i}\otimes\bigotimes_{j=1}^{m}X_{j}\right)\To \bigodot_{i=1}^{n}F(f_{i})\otimes\bigotimes_{j=1}^{m}F(X_{j})$$
defined by the comultiplication of the colax  monoidal functor $F$ is a weak equivalence in $\mor{\C{D}}$.
\end{lem}

\begin{proof}
Denote $f_{i}\colon U_{i}\into V_{i}$, $1\leq i\leq n$. 
The target of the morphism in the statement is the iterated comultiplication
$$F\left(\bigotimes_{i=1}^{n}V_{i}\otimes\bigotimes_{j=1}^{m}X_{j}\right)\To \bigotimes_{i=1}^{n}F(V_{i})\otimes\bigotimes_{j=1}^{m}F(X_{j}),$$
which is a weak equivalence by Lemma \ref{iterated}. The difficult part of this proof is to show that the source is also a weak equivalence. 

%By the push-out product axiom, the tensor product of cofibrant objects is cofibrant, and the tensor product of a cofibrant object and a cofibrantion is a cofibration. Hence

The diagram
$$\bigotimes_{i=1}^{n}f_{i}\otimes\bigotimes_{j=1}^{m}X_{j}\colon \dos^{n}\To\C{C}$$
is Reedy cofibrant in $\C{C}^{\dos^{n}}$ by Lemma \ref{reedy}. In particular,
$$F\left(\bigotimes_{i=1}^{n}f_{i}\otimes\bigotimes_{j=1}^{m}X_{j}\right)\colon \dos^{n}\To\C{D}$$
is Reedy cofibrant in $\C{D}^{\dos^{n}}$ by \cite[Proposition 15.4.1 (1)]{hirschhorn}, since left Quillen functors preserve cofibrant objects. Moreover, 
$$\bigotimes_{i=1}^{n}F(f_{i})\otimes\bigotimes_{j=1}^{m}F(X_{j})\colon \dos^{n}\To\C{D}$$
is also Reedy cofibrant in $\C{D}^{\dos^{n}}$ by Lemma \ref{reedy}. The restrictions to $\dos^{n}\setminus\{(1,\st{n}\dots,1)\}$ are also Reedy cofibrant, see \cite[Lemma 15.3.7. (1)]{hirschhorn}. 

Let $$\tau\colon F\left(\bigotimes_{i=1}^{n}f_{i}\otimes\bigotimes_{j=1}^{m}X_{j}\right)\Longrightarrow \bigotimes_{i=1}^{n}F(f_{i})\otimes\bigotimes_{j=1}^{m}F(X_{j})\colon \dos^{n}\To\C{D}$$
be the natural transformation defined pointwise by the iterated comultiplication weak equivalence in Lemma \ref{iterated}. 
The categories $\dos^{n}$ and $\dos^{n}\setminus\{(1,\st{n}\dots,1)\}$ are direct in the sense of \cite[Definition 5.1.1]{hmc}. Therefore,
$\colim_{\dos^{n}\setminus\{(1,\st{n}\dots,1)\}}\tau$ is a weak equivalence by \cite[Corollary 5.1.6]{hmc}, since left Quillen functors preserve weak equivalences between cofibrant objects. This finishes the proof, since $\colim_{\dos^{n}\setminus\{(1,\st{n}\dots,1)\}}\tau$ is the source of the morphism in the statement, compare Remark \ref{iteracion}.
\end{proof}

\begin{rem}
Lemma \ref{previo} is obviously also true if we reorder  the tensor and push-out product factors, i.e.~we may have
$$(X_{1}\otimes f_{1}\odot f_{2})\odot (X_{2}\otimes f_{3}\otimes X_{3}\otimes X_{4})\odot f_{4}\odot f_{5}.$$
%This can be better noticed if we rewrite the tensor products as push-out products,
%$$ (\varnothing\into X_{1})\odot f_{1}\odot f_{2} \odot  (\varnothing\into X_{2})\odot f_{3}\odot (\varnothing\into X_{3})\odot (\varnothing\into X_{4}) \odot f_{4}\odot f_{5}.$$
\end{rem}

\begin{proof}[Proof of Proposition \ref{paco}]
If  $\mathcal O=\unit_{\circ_{\C{V}}}$ is the initial operad in $\C V$ then 
$F^{\operatorname{oper}}(\unit_{\circ_{\C{V}}})=\unit_{\circ_{\C{W}}}$ is the initial operad in $\C{W}$, since
 $F^{\operatorname{oper}}$ is a left adjoint. Moreover, the morphism 
$\chi_{\unit_{\circ_{\C{V}}}}\colon F(\unit_{\circ_{\C{V}}})\r \unit_{\circ_{\C{W}}}$ 
is the counit 
$F(\unit_{{\C{V}}})\r \unit_{{\C{W}}}$ of $F$ in arity $1$. This counit is a weak equivalence by Definition \ref{weak} (2), since $\unit_{\C{V}}$ is cofibrant. In arities $n\neq 1$, $\chi_{\unit_{\circ_{\C{V}}}}$ is  the identity morphism on the initial object. Hence $\chi_{\unit_{\circ_{\C{V}}}}$  is a weak equivalence.

Consider a push-out in $\operad{\C V}$,
$$\xymatrix{\mathcal{F}_{\C V} (U)\ar@{}[rd]|{\text{push}}\ar@{ >->}[r]^-{\mathcal{F}_{\C V} (f)}\ar[d]_-{g}&\mathcal{F}_{\C V} (V)\ar[d]^-{g'}\\
\mathcal{O}\ar@{ >->}[r]_-{f'}&\mathcal{P}}$$
with $f$ a cofibration in $\C{V}^{\mathbb N}$ and $\mathcal{O}$  a cofibrant operad such that $\chi_{\mathcal O}$ is a weak equivalence. We are going to show that  $\chi_{\mathcal P}$ is also a weak equivalence in $\C{W}^{\mathbb{N}}$.

Let 
$$\mathcal{O}\st{\bar g}\longleftarrow U\st{f}\into V$$
be the diagram adjoint to the left upper corner. 
We can suppose without loss of generality that  $U(n)$ is cofibrant for all $n\geq 0$. This happens, for instance, if $U$ is the underlying sequence of a cofibrant operad, such as $\mathcal{O}$, see Corollary \ref{tech1.5bis}. 
If this condition did not hold, we could replace $f$ with $\tilde f$ in the following push-out diagram in $\C{V}^{\mathbb{N}}$,
$$\xymatrix{ U\ar@{ >->}[r]^-{f}\ar[d]_-{\bar g}\ar@{}[rd]|{\text{push}}& V\ar[d]\\
\mathcal{O}\ar@{ >->}[r]_-{\tilde f}&\mathcal{O}\cup_UV}$$
Indeed, if $\bar g'\colon V\r \mathcal{P}$ is adjoint to $g'$ then
$$\xymatrix{\mathcal{F}_{\C{V}}(\mathcal{O})\ar@{}[rd]|{\text{push}}\ar@{ >->}[r]^-{\mathcal{F}_{\C{V}}(\tilde f)}\ar[d]_-{\text{adjoint to }1_{\mathcal{O}}}&\mathcal{F}_{\C{V}}(\mathcal{O}\cup_UV)\ar[d]^-{\text{adjoint to }(1_{\mathcal{O}},\bar g')}\\
\mathcal{O}\ar@{ >->}[r]_-{f'}&\mathcal{P}}$$
is also a push-out in $\operad{\C{V}}$. 

Recall from Remark \ref{pushfree}   the decomposition of the morphism of sequences underlying $f'$ as a transfinite (countable) composition. This transfinite composition consists of cofibrations by Lemma \ref{sequicof}. Since $F$ is a left Quillen functor, $F(f')$ is the transfinite composition of
$$F(\mathcal{O})=F(P_0)\into\cdots\into F(P_{t-1})\st{F(\varphi_t)}\into F(P_t)\into \cdots,$$
where $F(\varphi_t(n))$, $n\geq 0$,
is a push-out along the cofibration
\begin{equation}\label{monstruo2}
\coprod_{T}F\left(\bigodot_{v\in I^{e}(T)}f(\val{T}{v})\;\;\otimes\bigotimes_{w\in I^{o}(T)}\mathcal{O}(\val{T}{w})\right).
\end{equation}

The cofibrant operad $F^{\operatorname{oper}}(\mathcal{P})$ fits into the following push-out diagram
$$
\xymatrix@C=40pt{\mathcal{F}_{\C{W}}(F(U))\ar[r]^-{\mathcal{F}_{\C{W}}(F(f))}\ar[d]_-{F^{\operatorname{oper}}(g)}&\mathcal{F}_{\C{W}}(F(V))\ar[d]^-{F^{\operatorname{oper}}(g')}\\
F^{\operatorname{oper}}(\mathcal{O})\ar[r]_-{F^{\operatorname{oper}}(f')}&F^{\operatorname{oper}}(\mathcal{P})}
$$
Hence, the morphism of sequences underlying $F^{\operatorname{oper}}(f')$ is the transfinite composition of 
$$F^{\operatorname{oper}}(\mathcal{O})=\tilde P_0\into\cdots\into \tilde P_{t-1}\st{\tilde \varphi_t}\into \tilde P_t\into \cdots,$$
where $\tilde\varphi_t(n)$, $n\geq 0$,
is a push-out along the cofibration
\begin{equation}\label{monstruo3}
\coprod_{T}\bigodot_{v\in I^{e}(T)}F(f(\val{T}{v}))\;\;\otimes\bigotimes_{w\in I^{o}(T)}F^{\operatorname{oper}}(\mathcal{O}(\val{T}{w})).
\end{equation}

The morphism $\chi_{\mathcal P}$ is the colimit of a diagram
$$\xymatrix@C=20pt{
F(\mathcal{O})\ar@{=}[r]\ar[d]^{\sim}_{\chi_{\mathcal{O}}}&F(P_{0})\ar@{ >->}[r]\ar[d]_{\chi^{0}_{\mathcal{O}}}&\cdots\ar@{ >->}[r] &F(P_{t-1})\ar@{ >->}[r]^-{F(\varphi_t)}\ar[d]_{\chi^{t-1}_{\mathcal{O}}}&F(P_{t})\ar@{ >->}[r]\ar[d]_{\chi^{t}_{\mathcal{O}}}&\cdots\\
F^{\operatorname{oper}}(\mathcal{O})\ar@{=}[r]&\tilde P_{0}\ar@{ >->}[r]&\cdots \ar@{ >->}[r]&\tilde P_{t-1}\ar@{ >->}[r]_-{\tilde\varphi_t}&\tilde P_{t}\ar@{ >->}[r]&\cdots
}$$
The morphism $\chi^{t}_{\mathcal{O}}(n)$, $n\geq 0$, is obtained from the previous one by taking push-out of horizontal arrows in the following diagram,
$$\xymatrix{F(P_{t-1})(n)\ar[d]_{\chi^{t-1}_{\mathcal{O}}(n)}&\bullet\ar[l]\ar@{ >->}[r]^{\eqref{monstruo2}}\ar[d]^{\sim}&\bullet\ar[d]^{\sim}\\
\tilde P_{t-1}(n)&\bullet\ar[l]\ar@{ >->}[r]^{\eqref{monstruo3}}&\bullet}$$
The commutative square on the right is a weak equivalence  $\eqref{monstruo2}\st{\sim}\r \eqref{monstruo3}$ in $\mor{\C{W}}$ defined as follows. It is a coproduct $\coprod_{T}$ of weak equivalences between cofibrant objects. Each factor of the coproduct decomposes as
$$\xymatrix{\displaystyle F\left(\bigodot_{v\in I^{e}(T)}     f(\val{T}{v})\otimes     \bigotimes_{w\in I^{o}(T)}     \mathcal{O}(\val{T}{w})  \right)
\ar[d]^-{\sim}\\   
\displaystyle\bigodot_{v\in I^{e}(T)}   F(f(\val{T}{v}))\otimes   \bigotimes_{w\in I^{o}(T)}   F(\mathcal{O}(\val{T}{w}))
\ar[d]^-{\sim}\\   
\displaystyle\bigodot_{v\in I^{e}(T)}   F(f(\val{T}{v}))\otimes   \bigotimes_{w\in I^{o}(T)}    F^{\operatorname{oper}}(\mathcal{O})(\val{T}{w})}
$$
Here, the first arrow is the weak equivalence in Lemma \ref{previo}. The second weak equivalence is 
$$\bigodot_{v\in I^{e}(T)}   F(f(\val{T}{v}))\otimes   \bigotimes_{w\in I^{o}(T)}    \chi_{\mathcal{O}}(\val{T}{w}).$$
All these objects are indeed cofibrant, since  the cofibration $f$ has cofibrant  source and the sequences $F(\mathcal{O})$ and $F^{\operatorname{oper}}(\mathcal{O})$ are cofibrant  by Corollary  \ref{tech1.5bis}.

%This is also a weak equivalence, since the tensor product in $\C{W}$ of a cofibrant object with a weak equivalence between cofibrant objects is a weak equivalence by the push-out product axiom. 

The cube lemma \cite[Lemma 5.2.6]{hmc} shows that if $\chi_{\mathcal{O}}^{t-1}(n)$ is a weak equivalence then $\chi_{\mathcal{O}}^{t}(n)$ is also a weak equivalence. Since $\chi_{\mathcal{O}}^{0}=\chi_{\mathcal{O}}$ is a weak equivalence by hypothesis, we deduce that $\chi_{\mathcal{O}}^{t}$ is a weak equivalence for all $t\geq 0$. Therefore, the colimit $\chi_{\mathcal{P}}=\colim_{t}\chi_{\mathcal{O}}^{t}$ is a weak equivalence by \cite[Proposition 15.10.12 (1)]{hirschhorn}.

Now, a standard inductive argument shows that $\chi_{\mathcal O}$ is a weak equivalence for any $\mathcal{F}_{\C V}(I_{\mathbb N})$-cell complex $\mathcal O$, and hence for any cofibrant operad $\mathcal O$ by the usual retract argument. Since $\chi_{\mathcal O}$ is a natural morphism of sequences, in order the inductive argument to work we should check that any  $\mathcal{F}(I_{\mathbb N})$-cell complex  is a transfinite composition of cofibrations of sequences with cofibrant starting sequence. The starting sequence is the initial operad, which is cofibrant   %by Corollary \ref{tech1.5bis}
because tensor units are cofibrant, and bonding morphisms are cofibrations of sequences by Corollary \ref{yi}.
\end{proof}

\begin{proof}[Proof of Theorem \ref{oqe}]
Let $\varphi\colon F^{\operatorname{oper}}(\mathcal{O})\r\mathcal{P}$ be a morphism in $\operad{\C{W}}$ with fibrant target such that $\mathcal{O}$ is cofibrant in  $\operad{\C{V}}$. We must show that $\varphi$ is a weak equivalence if and only if its adjoint $\varphi'\colon\mathcal{O}\r G(\mathcal{P})$ is a weak equivalence. 

Consider  the following morphisms of sequences in $\C{W}$,
$$F(\mathcal{O})\mathop{\To}^{\chi_{\mathcal{O}}}_{\sim} F^{\operatorname{oper}}(\mathcal{O})\st{\varphi}\To\mathcal{P}.$$
Here $\chi_{\mathcal{O}}$ is a weak equivalence by Proposition \ref{paco}. Recall that a morphism of operads is a weak equivalence (resp.~fibration)  if and only if the underlying morphism of sequences if a weak equivalence (resp.~fibration). Hence $\varphi$ is a weak equivalence if and only if $\varphi\chi_{\mathcal{O}}$ is a weak equivalence. Moreover,  the components of $\mathcal P$ are fibrant. Furthermore, the components of $\mathcal O$ are cofibrant by Corollary \ref{tech1.5bis}. Therefore,  since $F\dashv G$ is a Quillen equivalence, 
 $\varphi\chi_{\mathcal{O}}$ is a weak equivalence if and only if its adjoint along $F\dashv G$, i.e.~the morphism of sequences underlying $\varphi'\colon \mathcal{O}\r G(\mathcal{P})$, is a weak equivalence. Hence we are done.
\end{proof}

\section{Left properness for operads}\label{lpprueba}

In this section we prove Theorem \ref{lp} and deduce a gluing lemma for weak equivalences in $\operad{\C V}$.

\begin{proof}[Proof of Theorem \ref{lp}]
Assume first that  $\psi$ fits into a push-out diagram in $\operad{\C V}$
$$\xymatrix{\mathcal{F} (U)\ar@{}[rd]|{\text{push}}\ar@{ >->}[r]^-{\mathcal{F} (f)}\ar[d]_-{g}&\mathcal{F} (V)\ar[d]^-{g'}\\
\mathcal{O}\ar@{ >->}[r]_-{\psi}&\mathcal{Q}}$$
where $f$ is a cofibration in $\C{V}^{\mathbb N}$. In this case, $\psi'$ fits into the following push-out diagram
$$\xymatrix{\mathcal{F} (U)\ar@{}[rd]|{\text{push}}\ar@{ >->}[r]^-{\mathcal{F} (f)}\ar[d]_-{\varphi g}&\mathcal{F} (V)\ar[d]^-{\varphi 'g'}\\
\mathcal{P}\ar@{ >->}[r]_-{\psi'}&\mathcal{P}\cup_{\mathcal{O}}\mathcal{Q}}$$
We can suppose that $U(n)$ is cofibrant for all $n\geq 0$, compare the proof of Proposition \ref{paco}.

Recall from Remark \ref{pushfree} that the morphisms of sequences underlying $\psi$ and $\psi'$ decompose as transfinite compositions of
$$\mathcal{O}=Q_0\into\cdots\into Q_{t-1}\st{\psi_t}\into Q_t\into \cdots,\qquad\mathcal{P}=R_0\into\cdots\into R_{t-1}\st{\psi_t'}\into R_t\into \cdots,$$
where $\psi_t(n)$, $n\geq 0$,
is a push-out along
\begin{equation}\label{monstruo4}
\coprod_{T}\bigodot_{v\in I^{e}(T)}f(\val{T}{v})\;\;\otimes\bigotimes_{w\in I^{o}(T)}\mathcal{O}(\val{T}{w}),
\end{equation}
and  $\psi_t'(n)$, $n\geq 0$,
is a push-out along
\begin{equation}\label{monstruo5}
\coprod_{T}\bigodot_{v\in I^{e}(T)}f(\val{T}{v})\;\;\otimes\bigotimes_{w\in I^{o}(T)}\mathcal{P}(\val{T}{w}).
\end{equation}
All these morphisms are cofibrations by Lemma \ref{sequicof}. % Moreover, all these objects are cofibrant.

The morphism of sequences underlying $\varphi'$ is  the colimit of
$$\xymatrix@C=20pt{
\mathcal{O}\ar@{=}[r]\ar[d]^{\sim}_{\varphi}&Q_{0} \ar@{ >->}[r]\ar[d]_{\varphi_{0}}&\cdots\ar@{ >->}[r] &Q_{t-1} \ar@{ >->}[r]^-{\psi_t}\ar[d]_{\varphi_{t-1}}&Q_{t} \ar@{ >->}[r]\ar[d]_{\varphi_{t}}&\cdots\\
\mathcal{P}\ar@{=}[r]&R_{0}\ar@{ >->}[r]&\cdots \ar@{ >->}[r]&R_{t-1}\ar@{ >->}[r]_-{\psi_t'}&R_{t}\ar@{ >->}[r]&\cdots
}$$
Here, 
the morphism $\varphi_{t}(n)$, $n\geq 0$, is obtained from the previous one by taking push-out of horizontal arrows in the following diagram,
$$\xymatrix{Q_{t-1}(n)\ar[d]_{\varphi_{t-1}(n)}&\bullet\ar[l]\ar@{ >->}[r]^{\eqref{monstruo4}}\ar[d]^{\sim}&\bullet\ar[d]^{\sim}\\
R_{t-1}(n)&\bullet\ar[l]\ar@{ >->}[r]^{\eqref{monstruo5}}&\bullet}$$
The  square on the right is the following weak equivalence  $\eqref{monstruo4}\st{\sim}\r \eqref{monstruo5}$ in $\mor{\C{V}}$, $$\coprod_{T}\bigodot_{v\in I^{e}(T)}f(\val{T}{v})\;\;\otimes\bigotimes_{w\in I^{o}(T)}\varphi(\val{T}{w}).$$
This is indeed a weak equivalence, since each factor of the coproduct is a weak equivalence with cofibrant source and target. Here we use that the cofibration $f$ has cofibrant source and that the underlying sequences of $\mathcal O$ and $\mathcal P$ are assumed to be cofibrant.

%Indeed, each factor of this coproduct is a weak equivalence between cofibrant objects since this class of morphisms is closed under tensor products, by the push-out product axiom. 

The cube lemma \cite[Lemma 5.2.6]{hmc} shows that if $\varphi_{t-1}(n)$ is a weak equivalence then $\varphi_{t}(n)$ is also a weak equivalence. Since $\varphi_{0}=\varphi$ is a weak equivalence by hypothesis, we deduce that $\varphi_{t}$ is a weak equivalence for all $t\geq 0$. Therefore, the colimit $\varphi'=\colim_{t}\varphi_{t}$ is a weak equivalence by \cite[Proposition 15.10.12 (1)]{hirschhorn}.

Now, an inductive argument as in the proof of Proposition \ref{paco} shows that Theorem \ref{lp} is true if $\psi$ is a relative $\mathcal{F}(I_{\mathbb{N}})$-cell complex. We then deduce that Theorem \ref{lp} is true for any cofibration $\psi$ by the usual retract argument.
\end{proof}

Theorem \ref{lp} is an important ingredient to show that operads with underlying cofibrant sequence satify the axioms of a cofibration category in the sense of Baues \cite[I.1.1]{ah}.

\begin{prop}\label{cofcat}
Let $\C V$ be a symmetric monoidal model category. The full subcategory $\operadc{\C V}\subset \operad{\C V}$ spanned by the operads $\mathcal O$ such that $\mathcal O(n)$ is cofibrant for all $n\geq 0$ is a cofibration category.
\end{prop}

\begin{proof}
Axiom (C1) follows since $\operadc{\C V}$ is a full subcategory of a model category. Given a push-out diagram in $\operad{\C V}$,
$$\xymatrix{\mathcal{O}\ar@{ >->}[r]^-{\psi}\ar[d]_-{\varphi}\ar@{}[rd]|{\text{push}}&\mathcal{Q}\ar[d]^-{\varphi'}\\
\mathcal{P}\ar@{ >->}[r]_-{\psi'}&\mathcal{P}\cup_{\mathcal{O}}\mathcal{Q}}$$
if $\mathcal O$ and $\mathcal P$ are in $\operadc{\C V}$ then so are $\mathcal Q$ and $\mathcal{P}\cup_{\mathcal{O}}\mathcal{Q}$ by Corollary \ref{kisin}. Hence (C2) (b) follows since any push-out of a trivial cofibration in a model category is a trivial cofibration, and (C2) (b) follows from Theorem \ref{lp}. Axioms (C3) and (C4) also follow easily from Corollary \ref{kisin} and well known properties of model categories.
\end{proof}

The following result holds in any cofibration category, see \cite[II.1.2 (b)]{ah}.

\begin{cor}[Gluing lemma]\label{gluing}
Let $\C V$ be a symmetric monoidal model category. 
Consider a commutative diagram in $\operadc{\C V}$ as follows,
$$\xymatrix{\mathcal P\ar[d]^\sim&\mathcal O\ar[l]\ar[r]\ar[d]^\sim&\mathcal Q\ar[d]^\sim\\
\mathcal P'&\mathcal O'\ar[l]\ar[r]&\mathcal Q'}$$
If at least one of the morphisms in each row is a cofibration, then the induced morphism $\mathcal{P}\cup_{\mathcal{O}}\mathcal{Q}\r \mathcal{P}'\cup_{\mathcal{O}'}\mathcal{Q}'$ is a weak equivalence.
\end{cor}

\section{Algebras}

We here recall some facts about algebras over operads. See \cite{htnso} for details.

\begin{defn}
Let $\C{V}$ be a closed symmetric monoidal category with coproducts and $\C{C}$ a biclosed $\C{V}$-algebra with coproducts. An operad $\mathcal O$ in $\C V$ gives rise to a monad
$\mathcal{F}_{\mathcal O}\colon\C C\r \C C$ given by
$$\mathcal{F}_{\mathcal O}(X)=\coprod_{n\geq 0}z(\mathcal O(n))\otimes X^{\otimes n}.$$
The monad operations $\mathcal{F}_{\mathcal O}^{2}\rr \mathcal{F}_{\mathcal O}$ and $1_{\C C}\rr \mathcal{F}_{\mathcal O}$ are defined by the lax monoidal structure of $z$, see Definition \ref{valgebra}, and by the multiplications and the unit of the operad $\mathcal O$. 

An \emph{$\mathcal O$-algebra} in $\C C$ is an algebra over this monad. We denote $\algebra{\mathcal O}{\C C}$ the category of $\mathcal O$-algebras in $\C C$. 
\end{defn}

We also denote $\mathcal{F}_{\mathcal O}$ the free $\mathcal O$-algebra functor, left adjoint to the forgetful functor $\algebra{\mathcal O}{\C C}\r \C C$,
$$\xymatrix{\C{C}\ar@<.5ex>[r]^-{\mathcal{F}_{\mathcal O}}&\algebra{\mathcal O}{\C C}.\ar@<.5ex>[l]^-{\text{forget}}}$$

\begin{rem}
In a more explicit fashion, an $\mathcal O$-algebra is an object $A$ in $\C C$ together with structure morphisms, $n\geq 0$,
$$\nu_{n}^{A}\colon z(\mathcal O(n))\otimes A^{\otimes n}\To A$$
satisfying certain relations. An $\mathcal O$-algebra morphism $f\colon A\r B$ is a morphism in $\C C$ compatible with the structure morphisms. 
\end{rem}

\begin{rem}
The initial $\mathcal O$-algebra is $z(\mathcal O(0))$ with the following structure,
$$\xymatrix@C=40pt{\nu_{n}^{z(\mathcal O(0))}\colon z(\mathcal O(n))\otimes z(\mathcal O(0))^{\otimes n}\ar[r]^-{\text{mult. of }z}&
z(\mathcal O(n)\otimes \mathcal O(0)^{\otimes n})\ar[r]^-{z(\mu_{n;0,\dots,0})}&
z(\mathcal O(0)).
}$$

\end{rem}

\begin{rem}\label{pushfree2}
Suppose that the $\C{V}$-algebra $\C C$ is complete and cocomplete.  Then the category of algebras $\algebra{\mathcal O}{\C C}$ in $\C C$ over an operad $\mathcal O$ in $\C V$ is also complete and cocomplete. Limits and filtered colimits in $\algebra{\mathcal O}{\C{C}}$ are easy, since they are computed in $\C C$. Push-outs in $\algebra{\mathcal O}{\C{C}}$ are very complicated in general, but push-outs along free maps were carefully analyzed in \cite[\S 8]{htnso}, 
\begin{equation}\label{pusho2}
\xymatrix{\mathcal{F}_{\mathcal O} (U)\ar@{}[rd]|{\text{push}}\ar[r]^-{\mathcal{F}_{\mathcal O} (f)}\ar[d]_-{g}&\mathcal{F}_{\mathcal O} (V)\ar[d]^-{g'}\\
A\ar[r]_-{f'}&B}
\end{equation}
The  underlying morphism of $f'$ in $\C C$ is a transfinite (actually countable) composition of morphisms
$$A=B_0\st{\varphi_1}\To B_1\r\cdots\r B_{t-1}\st{\varphi_t}\To B_t\r \cdots,$$
where $\varphi_t\colon B_{t-1}\r B_{t}$
is a push-out along
\begin{equation}\label{killo}
\coprod\limits_{n\geq 1}\coprod\limits_{
\begin{array}{c}
\scriptstyle S\subset\{1,\dots,n\}\vspace{-3pt}\\
\scriptstyle \card(S)=t
\end{array}
}
z(\mathcal{O}(n))\otimes 
\bigodot_{S}f
\otimes\!\!\!\!\!\!
\bigotimes_{\{1,\dots,n\}\setminus S}\!\!\!\!\!\!A.
\end{equation}
Here we are abusing terminology, since the tensor and push-out product factors should be ordered according to the ordering of 
$\{1,\dots,n\}$ if $\C C$ is non-symmetric. Nevertheless, we will keep this notation throughout this paper, so as to simplify the exposition. 
We will not recall the explicit definition of the map from the source of \eqref{killo} to $B_{t-1}$, see \cite[Lemma 8.1]{htnso} for details.
\end{rem}

\begin{rem}
 We showed in \cite[Theorem 1.2]{htnso} that if $\C C$ is a model $\C V$-algebra with sets of generating cofibrations $I'$ and generating trivial cofibrations $J'$ with presentable sources, then $\algebra{\mathcal O}{\C C}$ has a cofibrantly generated model structure transferred along the free $\mathcal O$-algebra adjunction, i.e.~an $\mathcal O$-algebra morphism is a weak equivalence or fibration if and only if the underlying morphism in $\C C$ is so. Moreover,  $\mathcal F_{\mathcal O}(I')$ and $\mathcal F_{\mathcal O}(J')$  are sets of generating (trivial) cofibrations of $\algebra{\mathcal O}{\C C}$.
 \end{rem}
 
 The following result is similar to Lemma \ref{sequicof}.

\begin{lem}\label{sequicof2}
Consider a model $\C V$-algebra $\C C$. 
If $f$ is a cofibration in $\C{C}$, $\mathcal{O}$ is an operad in $\C{V}$ such that $z(\mathcal{O}(n))$ is cofibrant for all $n\geq 0$, and $A$ is an $\mathcal O$-algebra with underlying cofibrant object in $\C C$, then  \eqref{killo} is a cofibration in $\C{C}$.
\end{lem}

In practical examples, $z(\mathcal{O}(n))$ is cofibrant because $\mathcal{O}(n)$ is cofibrant, $n\geq 0$.

We recall \cite[Theorem 1.3]{htnso}.

\begin{thm}\label{coc}
Let $\C C$ be a model $\C V$-algebra. Consider a weak equivalence $\varphi\colon \mathcal O\st{\sim}\r\mathcal P$ between operads in $\C V$ such that the objects $ \mathcal O(n)$ and $\mathcal P(n)$ are cofibrant for all $n\geq 0$. Then the change of coefficients Quillen pair
$$\xymatrix{\algebra{\mathcal O}{\C{C}}\ar@<.5ex>[r]^-{\varphi_{*}}&\algebra{\mathcal P}{\C{C}}\ar@<.5ex>[l]^-{\varphi^{*}}}$$
is a Quillen equivalence. 
\end{thm}

In  \cite[Theorem 1.3]{htnso} we further assumed $\C C$ to be left proper, but this is actually unnecessary since the cube lemma  \cite[Lemma 5.2.6]{hmc} applies where we invoked left properness, in the proof of \cite[Lemma 9.6]{htnso}. We should also remark that in that proof we can suppose, without loss of generality, that $Y$ is a cofibrant object in $\C C$, applying if necessary the same trick as in the proof of Proposition \ref{paco} above. This trick can be applied since, under the hypotheses of Theorem \ref{coc}, any cofibrant $\mathcal O$-algebra has a cofibrant underlying object in $\C C$, see \cite[Lemma 9.4]{htnso}.

This kind of result goes back to \cite[Theorem 4.4]{ahto}.

\section{Change of base category for algebras}\label{malichon}

In this section we prove Theorem \ref{qea} and Proposition \ref{casigual2}. Actually, we show more general versions of these results.

Throughout this section, we place ourselves in the following context. We have a diagram
$$\xymatrix{
\C{V}\ar@<.5ex>[r]^-{F}\ar[d]_{z_{\C C}}&\C{W}\ar@<.5ex>[l]^-{G}\ar[d]^{z_{\C D}}
\\
\C{C}\ar@<.5ex>[r]^-{\bar F}&\C{D}\ar@<.5ex>[l]^-{\bar G}}$$
where $\C C$ is a model $\C V$-algebra, $\C D$ is a model $\C W$-algebra, the  Quillen pairs  $F\dashv G$ and $\bar F\dashv \bar G$ are  colax-lax   (symmetric) monoidal  adjunctions, and there is a fixed natural transformation
\begin{equation}\label{tau}
\tau(X)\colon\bar Fz_{\C C}(X)\To z_{\C D}F(X)
\end{equation}
which is a weak equivalence for $X$ cofibrant.
This natural transformation induces another one by adjunction,  $$\xy
(-5,0)*+{\tau'(X)\colon z_{\C C}G(X)}="a",
(45,0)*+{\bar Gz_{\C D}FG(X)}="b",
(90,0)*+{\bar Gz_{\C D}(X),}="c"
\ar"a";"b"^-{\text{adjoint to }\tau(G(X))}
\ar"b";"c"^-{\bar Gz_{\C D}(\text{counit of }F\dashv G)}
\endxy$$ that we assume to be monoidal. Moreover, for any $X$ in $\C W$ and any $Y$ in $\C D$ the following diagram must commute:
$$\xymatrix{&z_{\C C}G(X)\otimes_{\C C}\bar G(Y)\ar[rd]^{\qquad\zeta_{\C C}(G(X),\bar G(Y))}\ar[ld]_{\tau'(X)\otimes_{\C C}1\quad}&\\
\bar Gz_{\C D}(X)\otimes\bar G(Y)\ar[d]_{\text{mult.}}&&\bar G(Y)\otimes_{\C C}z_{\C C}G(X)\ar[d]^{1\otimes_{\C C}\tau'(X)}\\
\bar G(z_{\C D}(X)\otimes Y)\ar[rd]_{\bar G(\zeta_{\C D}(X,Y))\qquad}&&\bar G(Y)\otimes_{\C C}\bar Gz_{\C D}(X)\ar[ld]^{\text{mult.}}\\
&\bar G(Y\otimes_{\C D}z_{\C D}(X))&}$$
At the end of this section we consider the special case when $\C V=\C W$ and $F=G$ is the identity functor, which often arises.

Given an operad $\mathcal O$ in $\C V$, the functor $\bar G$ lifts to a functor
$$\bar G\colon \algebra{F^{\operatorname{oper}}(\mathcal O)}{\C D}\To  \algebra{\mathcal O}{\C C}.$$
If $B$ is an $F^{\operatorname{oper}}(\mathcal O)$-algebra in $\C D$, the $\mathcal O$-algebra structure on $\bar G(B)$ is defined as follows. The morphism $\nu_{n}^{\bar G (B)}$, $n\geq 0$, is
$$\xymatrix{z_{\C  C}(\mathcal O(n))\otimes_{\C C} \bar G(B)^{\otimes_{\C C}n}\ar[d]_{z_{\C C}(\text{unit of }F^{\operatorname{oper}}\dashv G)\otimes_{\C C} 1}\\
z_{\C  C}(G(F^{\operatorname{oper}}(\mathcal O)(n)))\otimes_{\C C} \bar G(B)^{\otimes_{\C C}n}
\ar[d]_{\tau'(F^{\operatorname{oper}}(\mathcal O)(n))\otimes_{\C C} 1}\\
\bar G(z_{\C  D}(F^{\operatorname{oper}}(\mathcal O)(n)))\otimes_{\C C} \bar G(B)^{\otimes_{\C C}n}\ar[d]_{\text{mult. of }\bar G}\\
\bar G(z_{\C  D}(F^{\operatorname{oper}}(\mathcal O)(n))\otimes_{\C D}  B^{\otimes_{\C D}n})\ar[d]_{\bar G(\nu_{n}^{B})}\\
\bar G(B)}$$

\begin{prop}\label{qp2}
If $\mathcal O$ is an operad in $\C V$, then the functor $\bar G$ induces a Quillen pair
$$\xymatrix{
\algebra{\mathcal O}{\C C}\ar@<.5ex>[r]^-{\bar F_{\mathcal O}}&\algebra{F^{\operatorname{oper}}(\mathcal O)}{\C D}
\ar@<.5ex>[l]^-{\bar G}.
}$$
\end{prop}

\begin{proof}
Since $\bar G\colon \C D\r \C C$ takes (trivial) fibrations in $\C{D}$ to (trivial) fibrations in $\C{C}$, then so does the induced functor between categories of algebras. The left adjoint $\bar F_{\mathcal O}$ exists by abstract reasons, see \cite[Theorem 4.5.6]{borceux2}, therefore $\bar F_{\mathcal O}\dashv \bar G$ is a Quillen pair. Notice that $\bar F_{\mathcal O}$, in general, is not given by $\bar F$ on underlying objects, unless $F$ and $\bar F$ are strong and \eqref{tau} is a natural isomorphism.
\end{proof}

The following result is our first main theorem on Quillen equivalences between categories of algebras over operads. This result contains Theorem \ref{qea} (1).

\begin{thm}\label{qea2}
Suppose  $\bar F\dashv \bar G$ is a weak monoidal Quillen equivalence, $\C V$ and $\C W$ have cofibrant tensor units, and $\mathcal O$ is a cofibrant operad in $\C V$. Then  the Quillen pair $$\xymatrix{
\algebra{\mathcal O}{\C C}\ar@<.5ex>[r]^-{\bar F_{\mathcal O}}&\algebra{F^{\operatorname{oper}}(\mathcal O)}{\C D}
\ar@<.5ex>[l]^-{\bar G}
}$$
 is a Quillen equivalence. 
\end{thm}

The proof of this theorem is formally identical to the proof of Theorem \ref{oqe}, replacing Proposition \ref{paco} with Proposition \ref{paco2} below. Therefore we omit it.

The functor $\bar F_{\mathcal O}$ may not coincide with  $\bar F$  on underlying objects, but how different are they? The following proposition answers this question. On cofibrant algebras, they coincide up to homotopy.

Given an $\mathcal O$-algebra $A$ in $\C C$, denote by 
\begin{equation}\label{chia}
\chi_{A}\colon \bar F(A)\To \bar F_{\mathcal O}(A)
\end{equation}
the natural morphism in $\C D$ which is adjoint along $\bar F\dashv \bar G$ to the unit of $\bar F_{\mathcal O}\dashv \bar G$. 

\begin{prop}\label{paco2}
Suppose $\bar F\dashv \bar G$ is a weak monoidal Quillen adjunction, $\C V$ and $\C W$ have cofibrant tensor units, and $\mathcal O$ is a cofibrant operad in $\C V$. If $A$ is a cofibrant $\mathcal O$-algebra in $\C C$, then $\chi_{A}$ is a weak equivalence.
\end{prop}

\begin{proof}
If $A=z_{\C C}(\mathcal O(0))$ is the initial $\mathcal O$-algebra in $\C C$, then $\bar F_{\mathcal O}(A)=z_{\C D}(F^{\operatorname{oper}}(\mathcal O)(0))$ is the initial $F^{\operatorname{oper}}(\mathcal O)$-algebra in $\C D$ and $\chi_{A}$ is 
$$
\xymatrix@C=50pt{\bar F z_{\C C}(\mathcal O(0))\ar[r]^{\tau(\mathcal O(0))}& z_{\C D}F(\mathcal O(0))\ar[r]^-{z_{\C D}(\chi_{\mathcal O}(0))}&
z_{\C D}(F^{\operatorname{oper}}(\mathcal O)(0)).}
$$
The object $\mathcal O(0)$ is cofibrant by Corollary \ref{tech1.5bis}, hence $\tau(\mathcal O(0))$ is a weak equivalence. Moreover, the morphism $\chi_{\mathcal O}(0)$  is a weak equivalence  by Proposition \ref{paco}, and its source and target are cofibrant by Corollary \ref{tech1.5bis}. Therefore $z_{\C D}(\chi_{\mathcal O}(0))$ is also a weak equivalence.

Consider a push-out diagram in $\algebra{\mathcal O}{\C C}$
\begin{equation*}
\xymatrix{\mathcal{F}_{\mathcal O} (U)\ar@{}[rd]|{\text{push}}\ar@{ >->}[r]^-{\mathcal{F}_{\mathcal O} (f)}\ar[d]_-{g}&\mathcal{F}_{\mathcal O} (V)\ar[d]^-{g'}\\
A\ar@{ >->}[r]_-{f'}&B}
\end{equation*}
such that $A$ is a cofibrant $\mathcal O$-algebra, $f$ is a cofibration in $\C C$, and $\chi_{A}$ is a weak equivalence. We are going to prove that $\chi_{B}$ is also a weak equivalence.

We can suppose that $U$ is cofibrant in $\C C$. If not, we can apply the same trick as in the proof of Proposition \ref{paco}, since any cofibrant $\mathcal O$-algebra has an underlying cofibrant object in $\C C$, see \cite[Lemma 9.4]{htnso}.

By Remark \ref{pushfree2}, the underlying morphism of $\bar F (f')$ in $\C C$ decomposes as a transfinite (countable) composition of
$$\bar F (A)=\bar F (B_0)\into\cdots\into \bar F (B_{t-1})\st{\bar F (\varphi_t)}\into \bar F (B_t)\into \cdots,$$
where $\bar F(\varphi_t)$
is a push-out along the cofibration
\begin{equation}\label{killo2}
\coprod\limits_{n\geq 1}\coprod\limits_{
\begin{array}{c}
\scriptstyle S\subset\{1,\dots,n\}\vspace{-3pt}\\
\scriptstyle \card(S)=t
\end{array}
}
\bar F (z_{\C C}(\mathcal{O}(n))\otimes 
\bigodot_{S}f
\otimes\!\!\!\!\!\!
\bigotimes_{\{1,\dots,n\}\setminus S}\!\!\!\!\!\!A).
\end{equation}
This morphism is indeed a cofibration by Lemma \ref{sequicof2}. The hypotheses of  Lemma \ref{sequicof2} are satisfied by Corollary \ref{tech1.5bis} and \cite[Lemma 9.4]{htnso}.

The $F^{\operatorname{oper}}(\mathcal O)$-algebra $\bar F_{\mathcal O}(B)$ fits into the following push-out  in $\algebra{F^{\operatorname{oper}}(\mathcal O)}{\C D}$,
\begin{equation*}
\xymatrix@C=50pt{\mathcal{F}_{F^{\operatorname{oper}}(\mathcal O)} \bar F(U)\ar@{}[rd]|{\text{push}}\ar[r]^-{\mathcal{F}_{F^{\operatorname{oper}}(\mathcal O)}\bar F (f)}\ar[d]_-{\bar F_{\mathcal O}(g)}&\mathcal{F}_{F^{\operatorname{oper}}(\mathcal O)}\bar F (V)\ar[d]^-{\bar F_{\mathcal O}(g')}\\
\bar F_{\mathcal O}(A)\ar[r]_-{\bar F_{\mathcal O}(f')}&\bar F_{\mathcal O}(B)}
\end{equation*}
and $\bar F_{\mathcal O} (f')$ decomposes in $\C D$ as the transfinite composition of
$$\bar F_{\mathcal O} (A)=\tilde B_0\into\cdots\into \tilde B_{t-1}\st{\tilde \varphi_t}\into \tilde B_t\into \cdots,$$
where $\tilde\varphi_t$
is a push-out along the cofibration
\begin{equation}\label{killo3}
\coprod\limits_{n\geq 1}\coprod\limits_{
\begin{array}{c}
\scriptstyle S\subset\{1,\dots,n\}\vspace{-3pt}\\
\scriptstyle \card(S)=t
\end{array}
}
z_{\C D}(F^{\operatorname{oper}}(\mathcal{O})(n))\otimes 
\bigodot_{S}\bar F(f)
\otimes\!\!\!\!\!\!
\bigotimes_{\{1,\dots,n\}\setminus S}\!\!\!\!\!\!\bar F(A).
\end{equation}
This morphism is a cofibration by the same reason as \eqref{killo2} above.

The morphism $\chi_{B}$ is the colimit of
$$\xymatrix@C=20pt{
\bar F(A)\ar@{=}[r]\ar[d]^{\sim}_{\chi_{A}}&\bar F(B_{0})\ar@{ >->}[r]\ar[d]_{\chi^{0}_{A}}&\cdots\ar@{ >->}[r] &\bar F(B_{t-1})\ar@{ >->}[r]^-{\bar F(\varphi_t)}\ar[d]_{\chi^{t-1}_{A}}&\bar F(B_{t})\ar@{ >->}[r]\ar[d]_{\chi^{t}_{A}}&\cdots\\
\bar F_{\mathcal{O}}(A)\ar@{=}[r]&\tilde B_{0}\ar@{ >->}[r]&\cdots \ar@{ >->}[r]&\tilde B_{t-1}\ar@{ >->}[r]_-{\tilde\varphi_t}&\tilde B_{t}\ar@{ >->}[r]&\cdots
}$$
All these objects are cofibrant by \cite[Lemma 9.4]{htnso}. The morphism $\chi^{t}_{A}$ is inductively obtained by taking  push-out of horizontal arrows in the following diagram,
$$\xymatrix{\bar F(B_{t-1})\ar[d]_{\chi^{t-1}_{A}}&\bullet\ar[l]\ar@{ >->}[r]^{\eqref{killo2}}\ar[d]^{\sim}&\bullet\ar[d]^{\sim}\\
\tilde B_{t-1}&\bullet\ar[l]\ar@{ >->}[r]^{\eqref{killo3}}&\bullet}$$
The commutative square on the right is a coproduct of weak equivalences in $\mor{\C D}$ defined as follows,
$$\xymatrix{
\displaystyle\bar F (z_{\C C}(\mathcal{O}(n))\otimes 
\bigodot_{S}f
\otimes\!\!\!\!\!\!
\bigotimes_{\{1,\dots,n\}\setminus S}\!\!\!\!\!\!A)\ar[d]_{\text{comult. of }\bar F}^{\sim}\\
\displaystyle\bar F (z_{\C C}(\mathcal{O}(n)))\otimes 
\bigodot_{S}\bar F (f)
\otimes\!\!\!\!\!\!
\bigotimes_{\{1,\dots,n\}\setminus S}\!\!\!\!\!\!\bar F (A)\ar[d]_{\tau(\mathcal O(n))\otimes 1}^{\sim}\\
\displaystyle z_{\C D}F(\mathcal{O}(n))\otimes 
\bigodot_{S}\bar F (f)
\otimes\!\!\!\!\!\!
\bigotimes_{\{1,\dots,n\}\setminus S}\!\!\!\!\!\!\bar F (A)
\ar[d]_{z_{\C D}(\chi_{\mathcal O}(n))\otimes 1}^{\sim}\\
\displaystyle z_{\C D}(F^{\operatorname{oper}}(\mathcal{O})(n))\otimes 
\bigodot_{S}\bar F(f)
\otimes\!\!\!\!\!\!
\bigotimes_{\{1,\dots,n\}\setminus S}\!\!\!\!\!\!\bar F(A)
}$$
In order these morphisms to be weak equivalences, we need to check some cofibrancy hypotheses, see Proposition \ref{paco} and Lemma \ref{previo}. This is also necessary to ensure that the coproduct of these weak equivalences is a weak equivalence. These hypotheses hold since all objects involved in this diagram are cofibrant. Indeed, since tensor units in $\C V$ and $\C W$ are assumed to be cofibrant, the underlying sequence of a cofibrant operad, such as $\mathcal O$ and $F^{\operatorname{oper}}(\mathcal{O})$,  is cofibrant by Corollary \ref{tech1.5bis}, and the underlying object of a cofibrant algebra over a cofibrant operad is cofibrant by \cite[Lemma 9.4]{htnso}. Actually, any cellular algebra over a cofibrant operad is a transfinite composition of cofibrations in the underlying category with cofibrant initial term, by \cite[Proposition 9.2 (2)]{htnso}. These facts will also be implicitly used below.

The cube lemma \cite[Lemma 5.2.6]{hmc} shows that if $\chi_{A}^{t-1}$ is a weak equivalence then $\chi_{A}^{t}$ is also a weak equivalence. Since $\chi_{A}^{0}=\chi_{A}$ is a weak equivalence by hypothesis, we deduce that $\chi_{A}^{t}$ is a weak equivalence for all $t\geq 0$. Therefore, the colimit $\chi_{B}=\colim_{t}\chi_{A}^{t}$ is a weak equivalence by \cite[Proposition 15.10.12 (1)]{hirschhorn}.

Now, a standard inductive argument shows that $\chi_{A}$ is a weak equivalence for any $\mathcal{F}_{\mathcal O}(I')$-cell complex $A$, and hence for any cofibrant $\mathcal O$-algebra $A$ by the usual retract argument. 
\end{proof}

%\begin{proof}[Proof of Theorem \ref{qea2}]
%Let $\phi\colon \bar L_{\mathcal{O}}(A)\r B$ be a morphism in $\algebra{L^{\operatorname{oper}}(\mathcal O)}{\C{W}}$ with fibrant target such that $A$ is cofibrant in  $\algebra{\mathcal O}{\C{V}}$. We must show that $\phi$ is a weak equivalence if and only if its adjoint $\phi'\colon A\r \bar R(B)$ is a weak equivalence. 
%
%Consider  the following morphisms  in $\C{D}$,
%$$L(A)\mathop{\To}^{\chi_{A}}_{\sim} \bar L_{\mathcal{O}}(A)\st{\phi}\To B.$$
%Here $\chi_{A}$ is a weak equivalence by Proposition \ref{paco2}. Since model structures on algebras are created by the forgetful functor to the underlying category, $\phi$ is a weak equivalence if and only if $\phi\chi_{A}$ is a weak equivalence. Moreover,  the underlying object  of $B$ is fibrant. Furthermore, the underlying object of $A$ is cofibrant by \cite[Lemma 9.4]{htnso}. Therefore,  since $\bar L\dashv \bar R$ is a Quillen equivalence, 
% $\phi\chi_{A}$ is a weak equivalence if and only if its adjoint, i.e.~the underlying morphism of  $\phi'\colon A\r \bar R(B)$, is a weak equivalence. Hence we are done.
%\end{proof}

Notice that the first part of Proposition \ref{casigual2} is a corollary of the previous result. Actually, the following more general statement holds.

\begin{cor}\label{casigual2.1}
Suppose $\bar F\dashv \bar G$ is a weak monoidal Quillen adjunction and $\C V$ and $\C W$ have cofibrant tensor units. Consider a cofibrant operad $\mathcal O$ in $\C V$. For any $\mathcal O$-algebra $A$ in $\C C$ there is a natural isomorphism in $\ho\C D$,
\begin{align*}
\mathbb{L}\bar F(A)&\cong\mathbb{L}\bar F_{\mathcal{O}}(A).
\end{align*}
\end{cor}

Consider now an operad $\mathcal P$ in $\C W$. The functor $\bar G$ lifts to a functor
$$\bar G\colon \algebra{\mathcal P}{\C D}\To  \algebra{G(\mathcal P)}{\C C}.$$
If $B$ is a $\mathcal P$-algebra in $\C D$, the $G(\mathcal P)$-algebra structure on $\bar G(B)$ is defined as follows. The morphism $\nu_{n}^{\bar G (B)}$, $n\geq 0$, is
$$\xymatrix{
z_{\C C}(G(\mathcal P(n)))\otimes \bar G(B)^{\otimes n}\ar[d]^-{\tau'(\mathcal P(n))\otimes 1}\\
\bar G(z_{\C D}(\mathcal P(n)))\otimes \bar G(B)^{\otimes n}\ar[d]^-{\text{mult.}}\\
\bar G(z_{\C D}(\mathcal P(n))\otimes B)^{\otimes n}\ar[d]^{\bar G(\nu_{n}^{B})}\\
\bar G(B)
}$$

The following result  can be proved as Proposition \ref{qp2}.

\begin{prop}\label{qp3}
If $\mathcal P$ is an operad in $\C W$, then the functor $\bar G$ induces a Quillen pair
$$\xymatrix{
\algebra{G(\mathcal P)}{\C C}\ar@<.5ex>[r]^-{\bar F^{\mathcal P}}&\algebra{\mathcal P}{\C D}
\ar@<.5ex>[l]^-{\bar G}.
}$$
\end{prop}

Our second main theorem on Quillen equivalences between categories of algebras over operads  extends Theorem \ref{qea} (2).

\begin{thm}\label{qea3}
Suppose  $F\dashv G$ and $\bar F\dashv \bar G$ are weak (symmetric) monoidal Quillen equivalences,  $\C V$ and $\C W$ have cofibrant tensor units, $\mathcal P$ is a fibrant operad in $\C W$, and the underlying sequences of $\mathcal P$ and  $G(\mathcal P)$ are cofibrant. Then  $$\xymatrix{
\algebra{G(\mathcal P)}{\C C}\ar@<.5ex>[r]^-{\bar F^{\mathcal P}}&\algebra{\mathcal P}{\C D}
\ar@<.5ex>[l]^-{\bar G}
}$$ is a Quillen equivalence. 
\end{thm}

\begin{proof}
Let $\varphi\colon\mathcal O\st{\sim}\onto G(\mathcal P)$ be a cofibrant resolution in $\operad{\C V}$, and $\bar\varphi\colon F^{\operatorname{oper}}(\mathcal O)\st{\sim}\r\mathcal P$   the adjoint morphism in $\operad{\C W}$, which is a weak equivalence by Theorem \ref{oqe}. Consider the following diagram of Quillen pairs,
$$\xymatrix{
\algebra{\mathcal O}{\C C}\ar@<.5ex>[r]^-{\bar F_{\mathcal O}}\ar@<-.5ex>[d]_{\varphi_{*}}
\ar@{<-}@<.5ex>[d]^{\varphi^{*}}&
\algebra{F^{\operatorname{oper}}(\mathcal O)}{\C D}
\ar@<.5ex>[l]^-{\bar G}
\ar@<-.5ex>[d]_{\bar\varphi_{*}}
\ar@{<-}@<.5ex>[d]^{\bar\varphi^{*}}
\\
\algebra{ G(\mathcal P)}{\C C}\ar@<.5ex>[r]^-{\bar F^{\mathcal P}}&\algebra{\mathcal P}{\C D}\ar@<.5ex>[l]^-{\bar G}}$$
The right adjoints commute, $\varphi^{*}\bar G=\bar\varphi^{*}\bar G$,  hence the left adjoints commute up to natural isomorphism, $\bar F^{\mathcal P}\varphi_{*}\cong\bar\varphi_{*}\bar F_{\mathcal O}$. The vertical Quillen pairs are Quillen  equivalences by Theorem \ref{coc}. Here we use that under our assumptions any cofibrant operad has an underlying cofibrant sequence, see Corollary \ref{tech1.5bis}. Moreover, the upper horizontal Quillen pair is  a Quillen equivalence by Theorem \ref{qea2}. Hence the bottom horizontal Quillen pair is also a Quillen equivalence.
\end{proof}

We finally deduce the following result, which generalizes the second part of Proposition \ref{casigual2}.

\begin{cor}\label{casigual2.2}
Suppose  $F\dashv G$ and $\bar F\dashv \bar G$ are weak (symmetric) monoidal Quillen equivalences and  $\C V$ and $\C W$ have cofibrant tensor units. Consider a fibrant operad $\mathcal P$ in $\C{W}$ such that $\mathcal P(n)$ and $G(\mathcal P(n))$ are cofibrant for all $n\geq 0$. For any $G(\mathcal P)$-algebra $B$ in $\C C$ there is a natural isomorphism in $\ho\C D$, 
\begin{align*}
\mathbb{L}\bar F(B)&\cong\mathbb{L}\bar F^{\mathcal{P}}(B).
\end{align*}
\end{cor}

\begin{proof}
Let $\chi^B\colon \bar F(B)\r \bar F^{\mathcal{P}}(B)$ be the adjoint of the unit of $\bar F^{\mathcal{P}}\dashv \bar G$ along $\bar F\dashv \bar G$. It is enough to show that $\chi^B$ is a weak equivalence for $B$ a cofibrant $G(\mathcal P)$-algebra in $\C C$. Let $q\colon \bar F^{\mathcal{P}}(B)\st{\sim}\into C$ be a fibrant replacement in $\algebra{\mathcal P}{\C D}$. By the 2-out-of-3 axiom, $\chi^B$ is a weak equivalence if and only if $q\chi^B$ is a weak equivalence. Since $\bar F\dashv\bar G$ is a Quillen equivalence, $q\chi^B$ is a weak equivalence if and only if its adjoint along $\bar F\dashv \bar G$ is a weak equivalence. The adjoint of $q\chi^B$ is a weak equivalence  by \cite[Proposition 1.3.13 (b)]{hmc}, since $\bar F^{\mathcal{P}}\dashv \bar G$ is a Quillen equivalence.
\end{proof}

In the special case that $\C V=\C W$ and $F\dashv G$ is the identity adjunction, Theorems \ref{qea2} and \ref{qea3} admit the following common generalization.

\begin{thm}\label{qea4}
Suppose  $\C V=\C W$ has a cofibrant tensor unit, $F=G=1_{\C V}$,  $\bar F\dashv \bar G$ is a weak monoidal Quillen equivalence, and $\mathcal O$ is an operad in $\C V$ with cofibrant underlying sequence. Then  $$\xymatrix{
\algebra{\mathcal O}{\C C}\ar@<.5ex>[r]^-{\bar F_{\mathcal P}=\bar F^{\mathcal P}}&\algebra{\mathcal O}{\C D}
\ar@<.5ex>[l]^-{\bar G}
}$$ is a Quillen equivalence. 
\end{thm}

\begin{proof}
Let $\varphi\colon\widetilde{\mathcal O}\st{\sim}\onto \mathcal O$ be a cofibrant resolution. Consider the following diagram of Quillen pairs,
$$\xymatrix{
\algebra{\widetilde{\mathcal O}}{\C C}\ar@<.5ex>[r]^-{\bar F_{\widetilde{\mathcal O}}=\bar F^{\widetilde{\mathcal O}}}\ar@<-.5ex>[d]_{\varphi_{*}}
\ar@{<-}@<.5ex>[d]^{\varphi^{*}}&
\algebra{\widetilde{\mathcal O}}{\C D}
\ar@<.5ex>[l]^-{\bar G}
\ar@<-.5ex>[d]_{\varphi_{*}}
\ar@{<-}@<.5ex>[d]^{\varphi^{*}}
\\
\algebra{\mathcal O}{\C C}\ar@<.5ex>[r]^-{\bar F_{{\mathcal O}}=\bar F^{\mathcal O}}&\algebra{\mathcal O}{\C D}\ar@<.5ex>[l]^-{\bar G}}$$
The right adjoints commute $\varphi^{*}\bar G=\bar G\varphi^{*}$, hence the left adjoints commute up to natural isomorphism. The vertical Quillen pairs are Quillen equivalences by Theorem \ref{coc}, since under our hypotheses any cofibrant operad has an underlying cofibrant sequence, see Corollary \ref{tech1.5bis}. Moreover, the upper horizontal Quillen pair is  a Quillen equivalence by Theorem \ref{qea2}. Hence the bottom horizontal Quillen pair is also a Quillen equivalence. 
\end{proof}

We omit the proof of the following corollary, which is very similar to the proofs of Corollaries \ref{casigual2.1} and \ref{casigual2.2}.

\begin{cor}\label{casigual2.3}
Under the assumptions of Theorem \ref{qea4}, for any $\mathcal O$-algebra $B$ in $\C C$ there is a natural isomorphism in $\ho\C D$, 
\begin{align*}
\mathbb{L}\bar F(B)&\cong\mathbb{L}\bar F_{\mathcal{O}}(B)=\mathbb{L}\bar F^{\mathcal{O}}(B).
\end{align*}
\end{cor}

\section{Left properness for algebras}\label{malichon2}

In this section we prove the following stronger version of Theorem \ref{lp3}. 

\begin{thm}\label{lp5}
Let $\C C$ be a model $\C{V}$-algebra. Consider an operad $\mathcal O$ in $\C V$ such that $z(\mathcal O(n))$ is cofibrant for all $n\geq 0$, and a push-out diagram in $\algebra{\mathcal O}{\C{C}}$ as follows,
$$\xymatrix{A\ar@{ >->}[r]^-{\psi}\ar[d]_-{\varphi}^{\sim}\ar@{}[rd]|{\text{push}}&C\ar[d]^-{\varphi'}\\
B\ar@{ >->}[r]_-{\psi'}&B\cup_{A}C}$$
If the underlying objects of $A$ and $B$ are cofibrant in $\C C$ then $\varphi'$ is a weak equivalence. 
\end{thm}

\begin{proof}
This proof follows the same pattern as the proof of Theorem \ref{lp} in Section \ref{lpprueba} above. 
Assume first that  $\psi$ fits into a push-out diagram in $\algebra{\mathcal O}{\C{C}}$
$$\xymatrix{\mathcal{F}_{\mathcal O} (U)\ar@{}[rd]|{\text{push}}\ar@{ >->}[r]^-{\mathcal{F}_{\mathcal O} (f)}\ar[d]_-{g}&\mathcal{F}_{\mathcal O} (V)\ar[d]^-{g'}\\
A\ar@{ >->}[r]_-{\psi}&C}$$
where $f$ is a cofibration in $\C{C}$. In this case, $\psi'$ fits into the following push-out diagram
$$\xymatrix{\mathcal{F}_{\mathcal O} (U)\ar@{}[rd]|{\text{push}}\ar@{ >->}[r]^-{\mathcal{F}_{\mathcal O} (f)}\ar[d]_-{\varphi g}&\mathcal{F}_{\mathcal O} (V)\ar[d]^-{\varphi 'g'}\\
B\ar@{ >->}[r]_-{\psi'}&B\cup_{A}C}$$
We can suppose that $U$ is cofibrant, compare the proof of Proposition \ref{paco}.

Recall from Remark \ref{pushfree2} that the morphisms  underlying $\psi$ and $\psi'$ decompose as transfinite compositions,
$$A=C_0\into\cdots\into C_{t-1}\st{\psi_t}\into C_t\into \cdots,\qquad B=D_0\into\cdots\into D_{t-1}\st{\psi_t'}\into D_t\into \cdots,$$
where $\psi_t$ 
is a push-out along
\begin{equation}\label{killo4}
\coprod\limits_{n\geq 1}\coprod\limits_{
\begin{array}{c}
\scriptstyle S\subset\{1,\dots,n\}\vspace{-3pt}\\
\scriptstyle \card(S)=t
\end{array}
}
z(\mathcal{O}(n))\otimes 
\bigodot_{S}f
\otimes\!\!\!\!\!\!
\bigotimes_{\{1,\dots,n\}\setminus S}\!\!\!\!\!\!A
,
\end{equation}
and  $\psi_t'$ 
is a push-out along
\begin{equation}\label{killo5}
\coprod\limits_{n\geq 1}\coprod\limits_{
\begin{array}{c}
\scriptstyle S\subset\{1,\dots,n\}\vspace{-3pt}\\
\scriptstyle \card(S)=t
\end{array}
}
z(\mathcal{O}(n))\otimes 
\bigodot_{S}f
\otimes\!\!\!\!\!\!
\bigotimes_{\{1,\dots,n\}\setminus S}\!\!\!\!\!\!B
.
\end{equation}
All these morphisms are cofibrations by Lemma \ref{sequicof2}. % Moreover, all these objects are cofibrant.

The morphism of sequences underlying $\varphi'$ is  the colimit of
$$\xymatrix@C=20pt{
A\ar@{=}[r]\ar[d]^{\sim}_{\varphi}&C_{0} \ar@{ >->}[r]\ar[d]_{\varphi_{0}}&\cdots\ar@{ >->}[r] &C_{t-1} \ar@{ >->}[r]^-{\psi_t}\ar[d]_{\varphi_{t-1}}&C_{t} \ar@{ >->}[r]\ar[d]_{\varphi_{t}}&\cdots\\
B\ar@{=}[r]&D_{0}\ar@{ >->}[r]&\cdots \ar@{ >->}[r]&D_{t-1}\ar@{ >->}[r]_-{\psi_t'}&D_{t}\ar@{ >->}[r]&\cdots
}$$
Here, 
the morphism $\varphi_{t}$ is obtained from the previous one by taking push-out of the horizontal arrows in the following diagram,
$$\xymatrix{C_{t-1}\ar[d]_{\varphi_{t-1}}&\bullet\ar[l]\ar@{ >->}[r]^{\eqref{killo4}}\ar[d]^{\sim}&\bullet\ar[d]^{\sim}\\
D_{t-1}&\bullet\ar[l]\ar@{ >->}[r]^{\eqref{killo5}}&\bullet}$$
The  square on the right is the following weak equivalence  $\eqref{killo4}\st{\sim}\r \eqref{killo5}$ in $\mor{\C{C}}$, $$\coprod\limits_{n\geq 1}\coprod\limits_{
\begin{array}{c}
\scriptstyle S\subset\{1,\dots,n\}\vspace{-3pt}\\
\scriptstyle \card(S)=t
\end{array}
}
z(\mathcal{O}(n))\otimes 
\bigodot_{S}f
\otimes\!\!\!\!\!\!
\bigotimes_{\{1,\dots,n\}\setminus S}\!\!\!\!\!\!\varphi
.$$
This is indeed a weak equivalence, since each factor of the coproduct is a weak equivalence with cofibrant source and target. 

%Indeed, each factor of this coproduct is a weak equivalence between cofibrant objects since this class of morphisms is closed under tensor products, by the push-out product axiom. 

The cube lemma \cite[Lemma 5.2.6]{hmc} shows that if $\varphi_{t-1}$ is a weak equivalence then $\varphi_{t}$ is also a weak equivalence. Since $\varphi_{0}=\varphi$ is a weak equivalence by hypothesis, we deduce that $\varphi_{t}$ is a weak equivalence for all $t\geq 0$. Therefore, the colimit $\varphi'=\colim_{t}\varphi_{t}$ is a weak equivalence by \cite[Proposition 15.10.12 (1)]{hirschhorn}.

Now, a standard  inductive argument shows that the statement is true for any relative $\mathcal{F}_{\mathcal O}(I')$-cell complex $\psi$, and hence  for any cofibration $\psi$ by the usual retract argument.
\end{proof}

The following corollary generalizes Corollary \ref{lp3.5}.

\begin{cor}\label{lp6}
Let $\C C$ be a model $\C{V}$-algebra. If all objects in $\C C$ are cofibrant then, for any operad $\mathcal O$ in $\C V$, the category $\algebra{\mathcal O}{\C{C}}$ is left proper.
\end{cor}

The following result can be proved as Proposition \ref{cofcat}, replacing Corollary   \ref{kisin} and Theorem \ref{lp} with \cite[Corollary 9.5]{htnso} and Theorem \ref{lp5}, respectively.

\begin{prop}\label{cofcat2}
Let $\C C$ be a model $\C V$-algebra and $\mathcal O$ an operad in $\C V$ such that $\mathcal O(n)$ is cofibrant for all $n\geq 0$. The full subcategory $\algebrac{\mathcal O}{\C C}\subset \algebra{\mathcal O}{\C C}$ spanned by the $\mathcal O$-algebras whose underlying object in $\C C$ is cofibrant is a cofibration category.
\end{prop}

In particular, the gluing lemma  also holds for $\algebrac{\mathcal O}{\C C}$ in these circumstances, compare Corollary \ref{gluing}.

\section{Homotopy invariance of the (unital) associative operad}\label{ainf}

Here we prove Theorem \ref{invariance}, Corollary \ref{invariance2}, and Propositions \ref{invariance3} and \ref{invariance4}.

Let $\C V$ be a closed symmetric monoidal category with coproducts. The \emph{unital associative operad}  $\mathtt{uAss}^{\C V}$ is defined by $\mathtt{uAss}^{\C V}(n)=\unit$ for all $n\geq 0$, the unit $\unit\r \mathtt{uAss}^{\C V}(1)$ is the identity morphism in $\unit$, and all composition laws are given by the unit isomorphism $\unit\otimes\unit\cong\unit$. Algebras over this operad are monoids. The \emph{associative operad}  $\mathtt{Ass}^{\C V}$ is identical to $\mathtt{uAss}^{\C V}$  except for arity zero, $\mathtt{Ass}^{\C V}(0)=\varnothing$. The operad structure is determined by the existence of an operad morphism $\phi^{\C V}\colon\mathtt{Ass}^{\C V}\r \mathtt{uAss}^{\C V}$ which is the identity in positive arities.  Algebras over this operad are non-unital monoids. Moreover, the functor $(\phi^{\C V})^{*}$ induced on algebras by the morphism $\phi^{\C V}$ is the usual forgetful functor from unital monoids to non-unital monoids. 

If $\C V$ is a symmetric monoidal model category, a \emph{(unital) $A$-infinity operad} $(\mathtt{u})\mathtt{A}^{\C V}_{\infty}$ is a cofibrant resolution of $(\mathtt{u})\mathtt{Ass}^{\C V}$ in $\operad{\C V}$. Here we mean cofibrant resolution in the following weak sense, we have weak equivalences with cofibrant sources,
\begin{align*}
q\colon \mathtt{uA}^{\C V}_{\infty}&\st{\sim}\To \mathtt{uAss}^{\C V},&
q'\colon \mathtt{A}^{\C V}_{\infty}&\st{\sim}\To \mathtt{Ass}^{\C V},
\end{align*}
we do not ask them to be fibrations. 
A \emph{(unital) $A$-infinity algebra} or \emph{(unital) $A$-infinity monoid} is a $(\mathtt{u})\mathtt{A}^{\C V}_{\infty}$-algebra.

Let $F\colon\C{V}\rightleftarrows\C{W}\colon G$ be a colax-lax symmetric monoidal adjunction between closed symmetric monoidal categories with coproducts. The (unital) associative operad is very special since the unit $\unit_{\C V}\r G(\unit_{\C W})$ of the lax symmetric monoidal functor $G$ induces morphisms in $\operad{\C V}$,
\begin{align*}%\label{ass}
\varphi\colon\mathtt{uAss}^{\C V}&\To G(\mathtt{uAss}^{\C W}),&
\varphi'\colon\mathtt{Ass}^{\C V}&\To G(\mathtt{Ass}^{\C W}).
\end{align*}

\begin{thm}\label{gordico}
 Let $F\colon\C{V}\rightleftarrows\C{W}\colon G$ be a weak symmetric monoidal Quillen adjunction. Suppose $\C V$ and $\C W$ have cofibrant tensor units. %Take cofibrant resolutions
% \begin{align*}
% q\colon \mathtt{uA}^{\C V}_{\infty}&\st{\sim}\onto \mathtt{uAss}^{\C V},&
% q'\colon \mathtt{A}^{\C V}_{\infty}&\st{\sim}\onto \mathtt{Ass}^{\C V},
% \end{align*}
% in $\operad{\C V}$. 
Then the morphisms
\begin{align*}%\label{assad}
\psi\colon F^{\operatorname{oper}}(\mathtt{uA}^{\C V}_{\infty})&\To \mathtt{uAss}^{\C W},&
\psi'\colon F^{\operatorname{oper}}(\mathtt{A}^{\C V}_{\infty})&\To \mathtt{Ass}^{\C W},
\end{align*}
adjoint to $\varphi q$ and $\varphi'q'$, respectively,  are weak equivalences in $\operad{\C W}$.
\end{thm}

\begin{proof}
We use the natural morphism \eqref{chio}.
For any $n\geq 0$, the composite morphism 
$$\xymatrix@C=40pt{F(\mathtt{uA}^{\C V}_{\infty}(n))\ar[r]^-{\chi_{\mathtt{uA}^{\C V}_{\infty}}(n)}&
F^{\operatorname{oper}}(\mathtt{uA}^{\C V}_{\infty})(n)\ar[r]^-{\psi(n)}&
\mathtt{uAss}^{\C W}(n)=\unit_{\C W}}$$
coincides with
$$\xymatrix@C=45pt{F(\mathtt{uA}^{\C V}_{\infty}(n))\ar[r]^-{F(q(n))}&
F(\mathtt{uAss}^{\C V}(n))=F(\unit_{\C V})\ar[r]^-{\text{counit of }F}&
\mathtt{uAss}^{\C W}(n)=\unit_{\C W}.}$$
The counit of $F$ is a weak equivalence by Definition \ref{weak} (2), and $F(q(n))$ is a weak equivalence since $q(n)$ is a weak equivalence between cofibrant objects, see Corollary \ref{tech1.5bis}. Moreover, $\chi_{\mathtt{uA}^{\C V}_{\infty}}(n)$ is a weak equivalence by Proposition \ref{paco}. Hence, $\psi(n)$ is a weak equivalence by the 2-out-of-3 axiom.

One can similarly check that $\psi'(n)$ is a weak equivalence for $n>0$. In arity zero, 
$$\xymatrix@C=40pt{F(\mathtt{A}^{\C V}_{\infty}(0))\ar[r]^-{\chi_{\mathtt{A}^{\C V}_{\infty}}(0)}&
F^{\operatorname{oper}}(\mathtt{A}^{\C V}_{\infty})(0)\ar[r]^-{\psi'(0)}&
\mathtt{Ass}^{\C W}(0)=\varnothing}$$
coincides with
$$\xymatrix@C=45pt{F(\mathtt{A}^{\C V}_{\infty}(0))\ar[r]^-{F(q'(0))}&
F(\mathtt{Ass}^{\C V}(0))=F(\varnothing)=\varnothing=
\mathtt{Ass}^{\C W}(0).}$$
This last morphism is a weak equivalence since $q'(0)$ is a weak equivalence between cofibrant objects, and $\chi_{\mathtt{A}^{\C V}_{\infty}}(0)$ is a weak equivalence by Proposition \ref{paco}. Hence, $\psi'(0)$ is also a weak equivalence by the 2-out-of-3 axiom.
\end{proof}

Theorem \ref{invariance} is a corollary of the previous result, since the following square commutes,
$$\xymatrix{
 \mathtt{Ass}^{\C V} \ar[r]^-{\phi^{\C V} }\ar[d]_-{\varphi'}& \mathtt{uAss}^{\C V}\ar[d]^-{\varphi}\\
  G(\mathtt{Ass}^{\C V}) \ar[r]^-{G(\phi^{\C V}) }& G(\mathtt{uAss}^{\C V}) }$$
%Corollary \ref{invariance2} needs no further explanation, it is a direct consequence of Theorem \ref{invariance}.

Monoids and $A$-infinity monoids, unital and non-unital, have Quillen equivalent model categories. In other words, $A$-infinity monoids can be rectified.

\begin{prop}\label{coc1.1}
Let $\C C$ be a model $\C V$-algebra. Assume that the tensor unit of $\C V$ is cofibrant. Then $q$ and $q'$ 
% , cofibrant resolutions
% \begin{align*}
% q\colon \mathtt{uA}^{\C V}_{\infty}&\st{\sim}\onto \mathtt{uAss}^{\C V},&
% q'\colon \mathtt{A}^{\C V}_{\infty}&\st{\sim}\onto \mathtt{Ass}^{\C V},
% \end{align*}
% in $\operad{\C V}$ 
induce Quillen equivalences
$$\xymatrix{\algebra{\mathtt{uA}^{\C V}_{\infty}}{\C{C}}\ar@<.5ex>[r]^-{q_{*}}&\algebra{\mathtt{uAss}^{\C V}}{\C{C}},\ar@<.5ex>[l]^-{q^{*}}}\qquad
\xymatrix{\algebra{\mathtt{A}^{\C V}_{\infty}}{\C{C}}\ar@<.5ex>[r]^-{q_{*}'}&\algebra{\mathtt{Ass}^{\C V}}{\C{C}}.\ar@<.5ex>[l]^-{(q')^{*}}}$$
\end{prop}

This is a corollary of Theorem \ref{coc}. It allows the proof of the following corollary, from which Propositions \ref{invariance3} and \ref{invariance4} follow. %, along the lines of Theorem \ref{qea3}. 

\begin{cor}\label{qea3.1}
Let us place ourselves in the situation described in  the first paragraph of Section \ref{malichon}. 
Suppose  that $F\dashv G$ and $\bar F\dashv \bar G$ are weak (symmetric) monoidal Quillen equivalences, and that $\C V$ and $\C W$ have cofibrant tensor units. Then,  there are Quillen equivalences,
$$\xymatrix{
\algebra{\mathtt{uAss}^{\C V} }{\C C}\ar@<.5ex>[r]&\algebra{\mathtt{uAss}^{\C W} }{\C D},
\ar@<.5ex>[l]^-{\bar G}&
\algebra{\mathtt{Ass}^{\C V} }{\C C}\ar@<.5ex>[r]&\algebra{\mathtt{Ass}^{\C W} }{\C D}
\ar@<.5ex>[l]^-{\bar G}.
}$$ 
\end{cor}

\begin{proof}
%Let $q\colon \mathtt{uA}^{\C V}_{\infty}\st{\sim}\onto \mathtt{uAss}^{\C V}$  be a cofibrant resolution in $\operad{\C V}$.  
Consider the following diagram of Quillen pairs
$$\xymatrix{
\algebra{\mathtt{uA}^{\C V}_{\infty} }{\C C}\ar@<.5ex>[r]^-{\bar F_{\mathtt{uA}^{\C V}_{\infty}}}
\ar@<.5ex>[d]^{(q_{\C V})_{*}}\ar@<-.5ex>@{<-}[d]_{q_{\C V}^{*}}
&
\algebra{F^{\operatorname{oper}}(\mathtt{uA}^{\C W}_{\infty}) }{\C D}
\ar@<.5ex>[l]^-{\bar G}
\ar@<.5ex>[d]^{\psi_{*}}\ar@<-.5ex>@{<-}[d]_{\psi^{*}}
\\
\algebra{\mathtt{uAss}^{\C V} }{\C C}\ar@<.5ex>[r]&\algebra{\mathtt{uAss}^{\C W} }{\C D}
\ar@<.5ex>[l]^-{\bar G}
}$$ 
The bottom $\bar G$ is the usual functor between categories of monoids induced by a lax monoidal functor.  This functor preserves (trivial) fibrations and has a left adjoint by abstract reasons, hence it is a right Quillen functor. 

The square formed by the right adjoints commutes, so left adjoints commute up to a natural isomorphism. 

The left vertical pair is a  Quillen equivalence by Proposition \ref{coc1.1}. The right vertical pair is a  Quillen equivalence by Theorems \ref{gordico} and \ref{coc}. The upper horizontal pair is a Quillen equivalence by Theorem \ref{qea2}. Hence, the bottom Quillen pair is also a Quillen equivalence. 

The proof of the non-unital case is formally the same.
\end{proof}

\appendix

\section{Pseudo-cofibrant objects}\label{pco}

In the series of appendices that we now start, we develop technical tools to avoid unnecessary cofibrancy hypotheses on tensor units. With this purpose, in this first appendix we isolate a certain property of cofibrant objects. The objects satisfying this property are called pseudo-cofibrant. We show that there are many pseudo-cofibrant objects, including the tensor unit, which are possibly non-cofibrant. Moreover, we prove that pseudo-cofibrant objects share many other features with cofibrant objects. 

\begin{defn}
A \emph{pseudo-cofibrant object} $X$ in a monoidal model category $\C{C}$ is an object such that
given a cofibration $f\colon U\into V$, the morphisms $f\otimes X$ and $X\otimes f$ are cofibrations.
\end{defn}

\begin{rem}\label{ejemplos}
Some examples of pseudo-cofibrant objects are:
\begin{enumerate}
\item Cofibrant objects,  by the push-out product axiom.
\item The tensor unit $\unit$.
\item Coproducts of pseudo-cofibrant objects.
\item Tensor products of pseudo-cofibrant objects.
\end{enumerate}
Tensoring with pseudo-cofibrant objects also preserves trivial cofibrations, by the monoid axiom. Therefore, pseudo-cofibrant objects are those objects $X$ such that $-\otimes X$ and $X\otimes -$ are left Quillen functors.
\end{rem}

As it usually happens in cofibrantly generated model categories, it is enough to check the property defining pseudo-cofibrant objects on generating cofibrations.

\begin{lem}\label{sufise}
An object $X$ is pseudo-cofibrant if and only if $f\otimes X$ and $X\otimes f$ are cofibrations for any generating cofibration $f$.
\end{lem}

\begin{proof}
The `only if' part is trivial. For the `if' part, assume $X\otimes-$ takes generating cofibrations to cofibrations. The monoid axiom shows that it also takes generating trivial cofibrations to cofibrations. Hence it is a left Quillen functor \cite[Lemma 2.1.20]{hmc}. Similarly for $-\otimes X$. Therefore $X$ is pseudo-cofibrant by the previous remark.
\end{proof}

The following result is a source of examples of pseudo-cofibrant objects.

\begin{lem}\label{cascajo}
If $g\colon X\into Y$ is a cofibration and $X$ is pseudo-cofibrant then $Y$ is also pseudo-cofibrant.
\end{lem}

\begin{proof}
Let $f\colon U\into V$ be a   cofibrantion. We are going to prove that $f\otimes Y$ is a   cofibration. One can similarly check that $Y\otimes f$ is also a   cofibration.

Consider the diagram \eqref{pushoutproduct} for the construction of $f\odot g$. The morphism $f\otimes X$ is a   cofibration since $X$ is pseudo-cofibrant, hence $\bar f$ is also a   cofibration. Moreover, $f\odot g$ is a   cofibration by the push-out product axiom, so $f\otimes Y=(f\odot g)\bar{f}$ is a   cofibration.
\end{proof}

\begin{cor}\label{cascajo2}
Given morphisms $Y\l X\into Z$, if $Y$ is pseudo-cofibrant then so is $Y\cup_{X}Z$.
\end{cor}

\begin{proof}
Simply recall that the canonical morphism $Y\into Y\cup_{X}Z$ is a cofibration.
\end{proof}

\begin{cor}\label{soncof2}
If  $f\colon U\into V$ and $g\colon X\into Y$ are  cofibrations  between pseudo-cofibrant objects, then the cofibration $f\odot g$ has pseudo-cofibrant source, and hence target.
\end{cor}

\begin{proof}
The source of $f\odot g$ is the push-out of  $U\otimes Y\st{U\otimes g}\leftarrowtail U\otimes X\st{f\otimes X}\into V\otimes X$. These three objects are pseudo-cofibrant, therefore the previous corollary applies.
\end{proof}

Pseudo-cofibrant objects become cofibrant if we tensor them with an honestly cofibrant object.

\begin{lem}\label{soncof}
Given two objects $U$ and $X$, if one of them is cofibrant and the other one is pseudo-cofibrant then $U\otimes X$ is cofibrant.
\end{lem}

\begin{proof}
If $U$ is the cofibrant object then $(\varnothing\into U)\otimes X=\varnothing\into U\otimes X$ is a cofibration, i.e.~$U\otimes X$ is cofibrant, and similarly if $X$ is the cofibrant object.
\end{proof}

\begin{cor}\label{soncof3}
If  $f\colon U\into V$ and $g\colon X\into Y$ are  cofibrations  between pseudo-cofibrant objects and $U$ or $X$ is cofibrant, then the cofibration $f\odot g$ has cofibrant source, and hence target.
\end{cor}

\begin{proof}
The source of $f\odot g$ is the push-out of  $U\otimes Y\st{U\otimes g}\leftarrowtail U\otimes X\st{f\otimes X}\into V\otimes X$, and these three objects are cofibrant by the previous lemma.
\end{proof}

\begin{defn}\label{mmc}
The following axiom may or may not be satisfied by a  {monoidal model category}:
\begin{itemize}
\item \emph{Strong unit axiom}: If $X$ is a pseudo-cofibrant object and $q\colon \tilde\unit\st{\sim}\onto\unit$ is a cofibrant resolution of the tensor unit then $X\otimes q$ and $q\otimes X$ are weak equivalences.

%\item \emph{Strongest unit axiom}: If $q\colon \tilde\unit\st{\sim}\onto\unit$ is a cofibrant resolution of the tensor unit an $f\colon X\st{\sim}\r Y$ is a weak equivalence between pseudo-cofibrant objects, then $\tilde\unit\otimes f$ and $f\otimes \tilde\unit$ are weak equivalences.
\end{itemize}
\end{defn}

\begin{rem}
If this axiom holds for a certain cofibrant resolution of $\unit$ then it holds for any cofibrant resolution of $\unit$. In particular, it is satisfied when $\unit$ is cofibrant. The following result shows that it also holds in monoidal model categories where tensoring with a cofibrant object preserves weak equivalences. This property is rather common.
\end{rem}

The proof of the following lemma is due to David White.

\begin{lem}\label{davidwhite}
Suppose that, for $q\colon\tilde\unit\st{\sim}\onto\unit$ a cofibrant resolution of the tensor unit, the functors $\tilde\unit\otimes -$ and $-\otimes\tilde\unit$ preserve weak equivalences between pseudo-cofibrant objects.  Then the strong unit axiom holds.
\end{lem}

\begin{proof}
Let $X$ be a pseudo-cofibrant object. Take 
a cofibrant resolution $p\colon \tilde X\st{\sim}\onto X$. Consider the following commutative diagram,
$$\xymatrix{\tilde\unit\otimes\tilde X\ar[r]^{\tilde\unit \otimes p}_{\sim}\ar[d]^{\sim}_{q\otimes\tilde X}&\tilde\unit\otimes X\ar[d]^{q\otimes X}\\
\tilde X\ar@{->>}[r]^{p}_{\sim}& X}$$
Here $q\otimes\tilde X$ is a weak equivalence by the unit axiom, and $\tilde\unit\otimes p$ is a weak equivalence by hypothesis. Hence $q\otimes X$ is a weak equivalence by the 2-out-of-3 axiom. One can similarly check that $X\otimes q$ is a weak equivalence.
\end{proof}

The converse  is also true, even something stronger holds.

\begin{lem}\label{chari}
If the strong unit axiom holds, a morphism between pseudo-cofibrant objects $f\colon U\r V$ is a weak equivalence if and only if $f\otimes\tilde\unit$ is a weak equivalence for some cofibrant replacement $\tilde\unit$ of the tensor unit. The same is true replacing $f\otimes\tilde\unit$ with $\tilde\unit\otimes f$.
\end{lem}

\begin{proof}
Let $q\colon\tilde\unit\st{\sim}\onto\unit$ be a cofibrant resolution of the tensor unit. Consider the following commutative diagram
$$\xymatrix{U\otimes \tilde\unit\ar[r]^{f\otimes \tilde\unit}\ar[d]_{U\otimes q}^\sim&V\otimes \tilde\unit\ar[d]^{V\otimes q}_\sim\\
 U\ar[r]^{f}& V}$$
The vertical morphisms are weak equivalences by the strong unit axiom. Hence, by the 2-out-of-3 axiom,  $f$ is a weak equivalence if and only if $f\otimes\tilde\unit$ is. Similarly for $\tilde\unit\otimes f$.
\end{proof}

\begin{lem}\label{pseudowe}
Let $\C{C}$ be a monoidal model category satisfying the strong unit axiom.
If $f\colon U\st{\sim}\r V$ is a weak equivalence with pseudo-cofibrant source and target and $X$ is a pseudo-cofibrant object, then $f\otimes X$ and $X\otimes f$ are weak equivalences.
\end{lem}

\begin{proof}
If $U$ and $V$ are cofibrant the result follows from the monoid axiom and Ken Brown's lemma  \cite[Lemma 1.1.12]{hmc}. In general, if $\tilde\unit$ is a cofibrant replacement of the tensor unit, $f\otimes\tilde\unit$ is a weak equivalence between cofibrant objects by Lemmas \ref{soncof} and \ref{chari}. Hence $X\otimes (f\otimes\tilde\unit)\cong(X\otimes f)\otimes\tilde\unit$ is a weak equivalence, so $X\otimes f$ is a weak equivalence between pseudo-cofibrant objects again by Lemma \ref{chari}, see Remark \ref{ejemplos} (4). One can similarly check that $f\otimes X$ is a weak equivalence.
\end{proof}

\begin{cor}\label{pseudowe1.2}
In a monoidal model category satisfying the strong unit axiom, weak equivalences between pseudo-cofibrant objects are closed under tensor products.
\end{cor}

%\begin{cor}\label{pseudowe1.25}
%Let $\C{C}$ be a monoidal model category satisfying the strong unit axiom.
%If $f\colon U\st{\sim}\r V$ and $g\colon X\st{\sim}\r Y$ are  weak equivalences with pseudo-presentable sources and targets  then $f\otimes g$ is a weak equivalence.
%\end{cor}

Lemma \ref{pseudowe} is a typical example of the following situation. We want to show that a certain morphism $\phi$ is a weak equivalence if some objects are pseudo-cofibrant, and the result happens to be known if those objects are honestly cofibrant. 
The pseudo-cofibrancy hypotheses and some of the previous results and remarks show that $\phi$ has pseudo-cofibrant source and target. We tensor everything with a cofibrant replacement $\tilde \unit$ of the tensor unit, to turn pseudo-cofibrant  objects into cofibrant objects by Lemma \ref{soncof}. The `only if' part of Lemma \ref{chari} shows that we are in the same situation than at the beginning, but now objects are cofibrant. Then we apply the known result to conclude that $\phi\otimes\tilde\unit$ or $\tilde\unit\otimes\phi$ is a weak equivalence, and we use the `if' part of Lemma \ref{chari}  to deduce that $\phi$ is indeed a weak equivalence. 

This technique can be used to prove the rest of the results in this section. Therefore we leave proofs as an excercise for the reader and content ourselves with indicating why the corresponding result is true when objects are cofibrant.

\begin{lem}\label{pseudowe1.5}
If $S$ is a set and $f_{s}$ are weak equivalences between pseudo-cofibrant objects in a monoidal model category satisfying the strong unit axiom, $s\in S$, then $\coprod_{s\in S}f_{s}$ is a weak equivalence between pseudo-cofibrant objects.
\end{lem}

The result for cofibrant objects follows from \cite[Proposition 7.2.5]{hirschhorn} and Ken Brown's lemma  \cite[Lemma 1.1.12]{hmc}.

%\begin{proof}
%The result is known if all objects are cofibrant. Hence, if $\tilde\unit$ is a cofibrant replacement of the tensor unit,  
%$\tilde\unit\otimes\coprod_{s\in S}f_{s}=\coprod_{s\in S}\tilde\unit\otimes f_{s}$
%is a weak equivalence, since $\tilde\unit\otimes f_{s}$ is a weak equivalence with cofibrant source and target by Lemmas \ref{soncof} and \ref{chari}, $s\in S$. Therefore $\coprod_{s\in S}f_{s}$ is a weak equivalence again by Lemma \ref{chari}.\end{proof}

The following result is a pseudo-cofibrant version of the cube lemma \cite[Lemma 5.2.6]{hmc}.

\begin{lem}[Cube lemma]\label{cube}
Let $\C{C}$ be a monoidal model category satisfying the strong unit axiom. Consider a commutative diagram as follows
$$\xymatrix{Y\ar[d]^\sim&X\ar[l]\ar@{ >->}[r]\ar[d]^\sim&Z\ar[d]^\sim\\
Y'&X'\ar[l]\ar@{ >->}[r]&Z'}$$
If all objects are pseudo-cofibrant then the induced morphism $\phi\colon Y\cup_XZ\r Y'\cup_{X'}Z'$ is a weak equivalence.
\end{lem}

The following lemma allows to weaken cofibrancy hypotheses in some inductive proofs, where we need to show that a certain natural morphism is a weak equivalence when evaluated at cofibrant objects.

\begin{lem}\label{colimit}
Let $\C{C}$ be a monoidal model category satisfying the strong unit axiom.
Consider an ordinal $\alpha$, two continuous functors $F,G\colon\alpha\r\C{C}$ and a natural transformation $\tau\colon F\r G$. Suppose that $F(0)$ and $G(0)$ are pseudo-cofibrant and, for $\beta<\alpha$, $\tau(\beta)$ is a weak equivalence and the morphisms $F(\beta)\r F(\beta+1)$ and $G(\beta)\r G(\beta+1)$ are cofibrations.
Then $\colim_{\beta<\alpha}\tau(\beta)$ is a  weak equivalence between pseudo-cofibrant objects.
\end{lem}

If $F(0)$ and $G(0)$ are cofibrant, the result follows from \cite[Corollary 5.1.6]{hmc} and Ken Brown's lemma  \cite[Lemma 1.1.12]{hmc}, since in this case $F$ and $G$ are cofibrant in the Reedy model category $\C{M}^\alpha$.

%\begin{proof}
%% All values of $F$ and $G$, and their colimits, are pseudo-cofibrant by Lemma \ref{cascajo}. 
%
%Let $\tilde\unit$ be a cofibrant replacement of the tensor unit. 
%The natural transformation $\tau\otimes\tilde\unit$ is a weak equivalence between cofibrant objects in the  Reedy model category $\C{M}^\alpha$, by the  Lemmas  \ref{soncof} and \ref{chari}. The colimit is a left Quillen functor, see  \cite[Corollary 5.1.6]{hmc}, hence $\colim_{\beta<\alpha}\tau(\beta)\otimes\tilde\unit$ is a weak equivalence. Now we apply Lemma \ref{chari} again to conclude that $\colim_{\beta<\alpha}\tau(\beta)$ is a weak equivalence.
%\end{proof}

\section{$\unit$-cofibrant objects}\label{ico}

In the previous appendix, we proved that pseudo-cofibrant objects share many properties with cofibrant objects. However, there is an important property which is not shared: left Quillen functors preserve cofibrant objects, but they need not preserve pseudo-cofibrant objects. This property of cofibrant objects is constantly (often implicitly) used in the proofs of our results on change of base category. In general, we do not know how all pseudo-cofibrant objects look like in a monoidal model category. Therefore, it is not reasonable that we restrict to work with Quillen pairs whose left adjoint preserves pseudo-cofibrant objects. We will rather work with left Quillen functors which send a smaller class of pseudo-cofibrant objects, easier to handle with and enough for our purposes, to pseudo-cofibrant objects.

\begin{defn}
An object $X$ in a monoidal model category $\C{C}$ is \emph{$\unit$-cofibrant} if there exists a cofibration $\unit\into X$.
\end{defn}

\begin{rem}
Notice that $\unit$-cofibrant objects are pseudo-cofibrant by Remark \ref{ejemplos} (2) and Lemma \ref{cascajo}. If $X\into Y$ is a cofibration and $X$ is $\unit$-cofibrant then so is $Y$. If $\unit$ is cofibrant, $\unit$-cofibrant objects are cofibrant.
\end{rem}

The following lemma admits the same proof as Corollary \ref{cascajo2}.

\begin{lem}\label{cascajo2.5}
Given morphisms $Y\l X\into Z$, if $Y$ is $\unit$-cofibrant then so is $Y\cup_{X}Z$.
\end{lem}

\begin{lem}\label{unos}
The tensor product of $\unit$-cofibrant objects is $\unit$-cofibrant. 
\end{lem}

\begin{proof}
Let $f\colon \unit\into U$ and $g\colon \unit\into X$ be cofibrations. The map $f\otimes g\colon \unit\otimes \unit\into U\otimes X$ is a cofibration since it decomposes as $f\otimes g=(f\otimes X)(\unit\otimes g)$, and $\unit\cong\unit\otimes \unit$, so $U\otimes X$ is $\unit$-cofibrant.
\end{proof}

The following corollary admits the same proof as Corollary \ref{soncof2}.

\begin{cor}\label{soncof2.5}
If  $f\colon U\into V$ and $g\colon X\into Y$ are  cofibrations  between $\unit$-cofibrant objects, then the cofibration $f\odot g$ has $\unit$-cofibrant source, and hence target.
\end{cor}

\begin{defn}\label{icofibrantax}
Let $F\colon\C{C}\r\C{D}$ be a functor between monoidal model categories. We consider the following axioms that $F$ may satisfy:
\begin{enumerate}
\item \emph{Pseudo-cofibrant axiom}: The object $F(\unit_{\C{C}})$ is pseudo-cofibrant in $\C{D}$.

\item \emph{$\unit$-cofibrant axiom}: If $p\colon \tilde X\st{\sim}\onto X$ is a cofibrant resolution of an $\unit_{\C{C}}$-cofibrant object $X$ in $\C{C}$, the morphism $F(p)$ is a weak equivalence in $\C{D}$.
\end{enumerate}
We say that a Quillen pair $F\colon\C{C}\rightleftarrows\C{D}\colon G$ satisfies one these axioms if the left adjoint $F$ does.
\end{defn}

\begin{rem}\label{hastio}
The pseudo-cofibrant axiom and the $\unit$-cofibrant axiom obviously hold when $\unit_{\C{C}}$ is cofibrant. The pseudo-cofibrant axiom also holds if $F$ is (co)lax monoidal and the (co)unit of $F$ is an isomorphism, e.g.~if $F$ is strong monoidal. If the $\unit$-cofibrant axiom holds for a certain cofibrant resolution of an $\unit$-cofibrant object $X$, then it holds for any cofibrant resolution of $X$. The $\unit$-cofibrant axiom is obviously satisfied if $F$ preserves weak equivalences.

These axioms may hold even for adjunctions $F\dashv G$ which do not satisfy compatibility properties with the monoidal structures, e.g. $\C C=\C D$ and $F=Y\otimes-$ the tensor product with a pseudo-cofibrant object $Y$ which is neither a monoid nor a comonoid, see Lemma \ref{pseudowe}.
\end{rem}

The following result is the motivation to introduce the pseudo-cofibrant axiom.

\begin{lem}\label{motivo}
If  $F\colon\C{C}\rightleftarrows\C{D}\colon G$ satisfies the pseudo-cofibrant axiom and $X$ is an $\unit_{\C{C}}$-cofibrant object in $\C{C}$ then $F(X)$ is pseudo-cofibrant in $\C{D}$.
\end{lem}

\begin{proof}
Since $F$ is a left Quillen functor, it sends a cofibration $\unit_{\C{C}}\into X$ in $\C{C}$ to a cofibration $F(\unit_{\C{C}})\into F(X)$  in $\C{D}$, hence $F(X)$ is pseudo-cofibrant by the pseudo-cofibrant axiom and  Lemma \ref{cascajo}.
\end{proof}

\begin{rem}
As we have previously said, left Quillen functors need not preserve pseudo-cofibrant objects in general. However, there is a specific kind of left Quillen functor which does, and we should record this fact here. 

We say that a Quillen adjunction $F\colon\C{C}\rightleftarrows\C{D}\colon G$, with $\C C$ cofibrantly generated by sets of generating (trivial) cofibrations  with presentable sources, \emph{creates} the model structure on $\C D$ if $G$ not only preserves, but also reflects fibrations and trivial fibrations. In this case, $\C D$ is cofibrantly generated and the image by $F$ of the set of generating (trivial) cofibrations of $\C C$ is a set of generating (trivial) cofibrations of~$\C D$.

Let $F\dashv G$ be a monoidal Quillen adjunction which creates the model structure on $\C D$. Then $F$ preserves pseudo-cofibrant objects. This follows from Lemma \ref{sufise}. Indeed, it is enough to check that $F(X)\otimes_{\C D} F(f)$ is a cofibration for any pseudo-cofibrant object $X$ and any generating cofibration $f$ in $\C C$. Since $F$ is strong, $F(X)\otimes_{\C D} F(f)\cong F(X\otimes_{\C C} f)$. Moreover,  $X\otimes_{\C C} f$ is a cofibration because $X$ is pseudo-cofibrant and $f$ is a cofibration. Hence, $F(X\otimes_{\C C} f)$ is a cofibration since $F$ is a left Quillen functor. 

%It does not seem reasonable to restrict below to monoidal Quillen adjunctions creating the model structure on the right. The pseudo-cofibrant axiom is a much weaker hypothesis, see Remark \ref{hastio}. However, this forces us to assume that some objects are cofibrant or $\unit$-cofibrant, rather than just pseudo-cofibrant.  Nevertheless, we wish record here that, if in the results below the Quillen adjunctions preserve pseudo-cofibrant objects, then it is enough to impose just pseudo-cofibrancy hypotheses.
\end{rem}

%In the presence of the $\unit$-cofibrant axiom, pseudo-cofibrant objects share more properties with cofibrant objects related to left Quillen functors.

\begin{lem}
Let  $F\colon\C{C}\rightleftarrows\C{D}\colon G$ be a Quillen equivalence satisfying the $\unit$-cofibrant  axiom. Consider a morphism $\phi\colon F(X)\r Y$ in $\C D$ where $Y$ is fibrant and $X$ is $\unit_{\C C}$-cofibrant. Then $\phi$ is a weak equivalence if and only if its adjoint $\phi'\colon X\r G(Y)$ is a weak equivalence.
\end{lem}

\begin{proof}
This would be true if $X$ were cofibrant. Take then a cofibrant resolution $p\colon\tilde X\st{\sim}\onto X$. By the $\unit$-cofibrant and the 2-out-of-3 axioms, $\phi$ is a weak equivalence if and only if $\phi F(p)$ is a weak equivalence. Since $F\dashv G$ is a Quillen equivalence, $\phi F(p)$ is a weak equivalence if and only if its adjoint, which is $\phi'p$, is a weak equivalence. Finally, by the 2-out-of-3 axiom, $\phi'p$ is a weak equivalence if and only if $\phi'$ is a weak equivalence.
\end{proof}

\begin{lem}\label{left}
A left Quillen functor $F$ satisfying the $\unit$-cofibrant axiom preserves weak equivalences between ($\unit$-)cofibrant objects.
\end{lem}

\begin{proof}
Let $f\colon U\st{\sim}\r V$ be such a weak equivalence. The result is known if $U$ and $V$ are cofibrant. In general, consier a commutative diagram
$$\xymatrix{
\tilde U\ar[r]^{\tilde f}_{\sim}\ar@{->>}[d]_{\sim}^{p}&\tilde V\ar@{->>}[d]^{\sim}_{q}\\
U\ar[r]^{f}_{\sim}&V}$$
where the vertical arrows are cofibrant resolutions. Take the image by $F$
$$\xymatrix{
F(\tilde U)\ar[r]^{F(\tilde f)}_{\sim}\ar[d]_{\sim}^{F(p)}&F(\tilde V)\ar[d]^{\sim}_{F(q)}\\
F(U)\ar[r]^{F(f)}&F(V)}$$
The vertical arrows are weak equivalences by the $\unit$-cofibrant axiom and $F(\tilde f)$ is a weak equivalence because $\tilde f$ has cofibrant source and target. Hence $F(f)$ is a weak equivalence by the 2-out-of-3 axiom.
\end{proof}

%The following characterization of the $\unit$-cofibrant axiom will be used more often than the original formulation.
%
%
%\begin{lem}\label{icofibrantaxeq}
%Suppose $\C{C}$ satisfies the strong unit axiom. 
%Then the {$\unit$-cofibrant axiom} is equivalent to the following statement: given an $\unit$-cofibrant object $X$ and a cofibrant resolution $q\colon \tilde \unit\st{\sim}\onto \unit$  of the tensor unit, the morphism $F(q\otimes X)$ is a weak equivalence. The same holds replacing $F(q\otimes X)$ with $F(X\otimes q)$. 
%\end{lem}
%
%\begin{proof}
%We here simply prove the first part of the statement. Consider the following diagram of countinuous arrows
%$$\xymatrix{\varnothing\ar[r]\ar@{ >->}[d]&\tilde X\ar@{->>}[d]^{p}_{\sim}\\
%\tilde\unit\otimes X\ar[r]_-{q\otimes X}^{\sim}\ar[r]\ar@{-->}[ru]^{f}&X}$$
%Here $q\otimes X$ is a weak equivalence by the strong unit axiom, and $\tilde\unit\otimes X$ is cofibrant by Lemma \ref{soncof}, hence a lifting $f$ exists and is a weak equivalence by the 2-out-of-3 axiom. Moreover, $F(f)$ is a weak equivalence since $f$ has cofibrant source and target and $F$ is a left Quillen functor. Therefore, again by the 2-out-of-3 axiom, $F(p)$ is a weak equivalence if and only if $F(q\otimes X)$ is a weak equivalence.
%\end{proof}

\begin{lem}\label{comult}
Let $F\colon\C{C}\rightleftarrows\C{D}\colon G$ be a weak monoidal Quillen adjunction  satisfying the pseudo-cofibrant and the $\unit$-cofibrant axioms. Suppose $\C{C}$ and $\C{D}$ satisfy the strong  unit axiom. The comultiplication 
$F(U\otimes X)\r F(U)\otimes F(X)$ is a weak equivalence if $U$ and $X$ are ($\unit$-)cofibrant.
\end{lem}

\begin{proof}
Take cofibrant resolutions $p\colon\tilde U\st{\sim}\onto U$ and $q\colon\tilde X\st{\sim}\onto X$. Consider the following commutative diagram,
$$\xymatrix@C=30pt{F(\tilde U\otimes\tilde  X)\ar[d]_{F(p\otimes q)}
\ar[r]^-{\text{comult.}}&
F(\tilde U)\otimes F(\tilde X)\ar[d]^{F(p)\otimes F(q)}\\
F(U\otimes X)\ar[r]^-{\text{comult.}}&F(U)\otimes F(X)}$$
The upper horizontal arrow is a weak equivalence since $F\dashv G$ is a weak monoidal Quillen adjunction and $\tilde U$ and $\tilde X$ are cofibrant. Moreover, $p\otimes q$ is a weak equivalence by Corollary \ref{pseudowe1.2}. The source and target of $p\otimes q$ are ($\unit$-)cofibrant by Lemmas \ref{soncof} and \ref{unos}, hence $F(p\otimes q)$ is a weak equivalence by Lemma \ref{left}. That lemma also shows that $F(p)$ and $F(q)$ are weak equivalences. These weak equivalences have pseudo-cofibrant source and target by Lemma \ref{motivo}, hence $F(p)\otimes F(q)$ is a weak equivalence by  Corollary \ref{pseudowe1.2}. Finally, the bottom horizontal arrow is a weak equivalence by the 2-out-of-3 axiom.
%
%
%
%If $U$ and $X$ are cofibrant, this is (1) in the definition of weak monoidal Quillen adjunction.
%If $U$ is cofibrant and $X$ is $\unit$-cofibrant, we consider the following commutative square,
%$$\xymatrix{F(U\otimes X\otimes\tilde\unit)\ar[d]_{F(U\otimes  X\otimes q)}
%\ar[r]^-{\text{comult}}&
%F(U)\otimes F(X\otimes\tilde\unit)\ar[d]^{F(U)\otimes F(X\otimes q)}\\
%F(U\otimes X)\ar[r]^-{\text{comult}}&F(U)\otimes F(X)}$$
%Here, the upper horizontal arrow is a weak equivalence since $U$ and $X\otimes\tilde\unit$ are cofibrant. The morphism $U\otimes X\otimes q$ is a weak equivalence with cofibrant source and target by the unit axiom and Lemma \ref{soncof}, hence $F(U\otimes X\otimes q)$ is a weak equivalence. Moreover, $F(X\otimes q)$ is a weak equivalence  by the characterization of the $\unit$-cofibrant axiom in Lemma \ref{icofibrantaxeq}, and $F(U)\otimes F(X\otimes q)$ is a weak equivalence by  Lemmas \ref{pseudowe} and \ref{motivo}. Finally, the 2-ou-of-3 axiom proves that the bottom horizontal arrow is also a weak equivalence. One can proceed similarly in the remaining two cases.
\end{proof}

The following iteration of the previous lemma can be proved along the lines of Lemma \ref{iterated} using Lemma \ref{pseudowe}.

\begin{lem}\label{iterated2}
Let be $F\colon\C{C}\rightleftarrows\C{D}\colon G$  be in the conditions of Lemma \ref{comult}. Given ($\unit$-)cofibrant objects $X_{1},\dots, X_{n}$  in $\C{C}$, $n\geq 1$,  the iterated comultiplication
$$F\left(\bigotimes_{i=1}^{n}X_{i}\right)\To \bigotimes_{i=1}^{n}F(X_{i})$$
is a weak equivalence in $\C{D}$.
\end{lem}

We also generalize Lemma \ref{previo}.

\begin{lem}\label{previo2}
Let $F\colon\C{C}\rightleftarrows\C{D}\colon G$ be in the conditions of Lemma \ref{comult}. Consider cofibrations with ($\unit$-)cofibrant source and target $f_{1},\dots, f_{n}$ and ($\unit$-)cofibrant objects $X_{1},\dots, X_{m}$ in $\C{C}$. The morphism in $\mor{\C{D}}$
$$F\left(\bigodot_{i=1}^{n}f_{i}\otimes\bigotimes_{j=1}^{m}X_{j}\right)\To \bigodot_{i=1}^{n}F(f_{i})\otimes\bigotimes_{j=1}^{m}F(X_{j})$$
defined by the comultiplications of the colax  monoidal functor $F$ is a weak equivalence in $\mor{\C{D}}$.
\end{lem}

\begin{proof}
The target is a weak equivalence by Lemma \ref{iterated2}. The difficult part, as in the proof of Lemma \ref{previo}, is to show that the source is a weak equivalence.

The the source is the comit of the restriction to $\dos^{n}\setminus\{(1,\st{n}\dots,1)\}$ of the natural transformation
 $$\tau\colon F\left(\bigotimes_{i=1}^{n}f_{i}\otimes\bigotimes_{j=1}^{m}X_{j}\right)\To \bigotimes_{i=1}^{n}F(f_{i})\otimes\bigotimes_{j=1}^{m}F(X_{j})\colon \dos^{n}\To\C{D}$$
defined by the comultiplication of $F$, which consists of weak equivalences by  Lemma \ref{iterated2}.  We cannot directly conclude that the colimit is a weak equivalence since the source and target of $\tau$ are not Reedy cofibrant. The problem is that $\tau(0,\st{n}\dots,0)$ need not have  cofibrant source and target, only pseudo-cofibrant by the pseudo-cofibrant axiom and Lemma \ref{motivo}. This problem can be overcome by applying Lemma \ref{chari} as indicated in the previous section.
\end{proof}

\section{Generalized results for operads}\label{geno}

In this appendix we present our main results on the homotopy theory of operads without cofibrancy hypotheses on tensor units. We rather assume (some of) the weaker axioms introduced in Appendices \ref{pco} and \ref{ico}. The proofs of these results go along the same lines as their analogues, using here the theory developed in the two previous appendices.

Lemma \ref{sequicof} is also true when the operad has a pseudo-cofibrant underlying sequence. 

\begin{lem}\label{sequicof10}
If $\C V$ is a symmetric monoidal model category, $f$ is a cofibration in $\C{V}^{\mathbb{N}}$, and $\mathcal{O}$ is an operad in $\C{V}$ such that $\mathcal{O}(n)$ is pseudo-cofibrant for all $n\geq 0$, then the morphism \eqref{monstruo} is a cofibration in $\C{V}$.
\end{lem}

We derive straightforward generalizations of Corollaries \ref{kisin} and \ref{tech1.5bis}.

\begin{cor}
Let $\C V$ be a symmetric monoidal model category. 
If $\phi\colon\mathcal{O}\into\mathcal{P}$ is a cofibration in $\operad{\C V}$ and $\mathcal{O}(n)$ is pseudo-cofibrant in $\C{V}$ for all $n\geq 0$, then $\phi(n)\colon \mathcal{O}(n)\r \mathcal{P}(n)$ is a cofibration for all $n\geq 0$, in particular $\mathcal{P}(n)$ is pseudo-cofibrant.
\end{cor}

\begin{cor}\label{tech1.5bis2}
Let $\C V$ be a symmetric monoidal model category. 
If $\mathcal{O}$ is cofibrant in $\operad{\C V}$ then $\mathcal{O}(1)$ is $\unit$-cofibrant and $\mathcal{O}(n)$ is cofibrant in $\C{V}$ for all $n\neq 1$.
\end{cor}

Proposition \ref{paco} holds under weaker hypotheses. 

\begin{prop}\label{pacofinal}
Let $F\colon\C{V}\rightleftarrows\C{W}\colon G$ be a weak symmetric monoidal Quillen adjunction between symmetric monoidal model categories satisfying the strong unit axiom. Suppose that $F\dashv G$ satisfies the pseudo-cofibrant axiom and the $\unit$-cofibrant axiom. If $\mathcal{O}$ is cofibrant in $\operad{\C V}$ then the natural morphism $\chi_ {\mathcal{O}}$ in \eqref{chio} is a weak equivalence in $\C{W}^{\mathbb{N}}$. % and $\C{W}$ is left proper. 
\end{prop}

\begin{cor}\label{casigualfinal}
Let $F\colon\C{V}\rightleftarrows\C{W}\colon G$ be as in the previous proposition. 
For any operad $\mathcal O$ in $\C V$, there are  natural isomorphisms in $\ho\C W$, $n\geq 0$, 
$$\mathbb{L}F(\mathcal O(n))\cong\mathbb{L}F^{\operatorname{oper}}(\mathcal O)(n).$$
\end{cor}

Proposition \ref{pacofinal} allows to prove the following generalization of Theorem \ref{oqe}.

\begin{thm}\label{oqefinal}
In the situation of the previous proposition, assume further that $F\dashv G$ is a Quillen equivalence.  Then the Quillen pair in Proposition \ref{qp}  is a Quillen equivalence
$$\xymatrix{\operad{\C{V}}\ar@<.5ex>[r]^-{F^{\operatorname{oper}}}&\operad{\C{W}}.\ar@<.5ex>[l]^-{G}}$$
\end{thm}

Theorem \ref{lp} is also valid if we replace the cofibrancy hypotheses with pseudo-cofibrancy hypotheses, provided the strong unit axiom holds.

\begin{thm}\label{lpfinal}
Let $\C{V}$ be a symmetric monoidal model category satisfying the strong unit axiom. Consider a push-out diagram in $\operad{\C{V}}$ as follows,
$$\xymatrix{\mathcal{O}\ar@{ >->}[r]^-{\psi}\ar[d]_-{\varphi}^{\sim}\ar@{}[rd]|{\text{push}}&\mathcal{Q}\ar[d]^-{\varphi'}\\
\mathcal{P}\ar@{ >->}[r]_-{\psi'}&\mathcal{P}\cup_{\mathcal{O}}\mathcal{Q}}$$
If $\mathcal{O}(n)$ and $\mathcal{P}(n)$ are pseudo-cofibrant in $\C{V}$ for all $n\geq 0$, then $\varphi'$ is a weak equivalence. 
\end{thm}

The following proposition is essentially a corollary of the previous theorem.

\begin{prop}
Let $\C V$ be a symmetric monoidal model category satisfying the strong unit axiom. The full subcategory $\operadpc{\C V}\subset \operad{\C V}$ spanned by the operads $\mathcal O$ such that $\mathcal O(n)$ is pseudo-cofibrant for all $n\geq 0$ is a cofibration category.
\end{prop}

In particular, the gluing lemma  also holds for $\operadpc{\C V}$ in these circumstances, compare Corollary \ref{gluing}.

\section{Generalized results for algebras}\label{gena}

In this appendix, we do for algebras what we have just done for operads in the previous appendix.

The first result is a version of Lemma \ref{sequicof2}. 

\begin{lem}\label{sequicof2final}
Let $\C C$ be a model $\C V$-algebra satisfying  the strong unit axiom. 
If $f$ is a cofibration in $\C{C}$, $\mathcal{O}$ is an operad in $\C{V}$ such that $z(\mathcal{O}(n))$ is pseudo-cofibrant for all $n\geq 0$, and $A$ is an $\mathcal O$-algebra with underlying pseudo-cofibrant object in $\C C$, then  \eqref{killo} is a cofibration in $\C{C}$.
\end{lem}

In practical examples, $z(\mathcal{O}(n))$ is pseudo-cofibrant because $\mathcal{O}(n)$ is ($\unit$-)cofibrant, $n\geq 0$. Recall from Remark \ref{hastio} that the left Quillen functor $z\colon\C V\r\C C$ satisfies the pseudo-cofibrant axiom since it is strong.

We derive straightforward generalizations of \cite[Lemma 9.4 and Corollary 9.5]{htnso}.

\begin{cor}
Let $\C C$ be a model $\C V$-algebra satisfying  the strong unit axiom. 
Consider  an operad  $\mathcal O$ in ${\C V}$ such that $z(\mathcal{O}(n))$ is pseudo-cofibrant for all $n\geq 0$.
If $\phi\colon A\into B$ is a cofibration in $\algebra{\mathcal O}{\C C}$ and $A$ is pseudo-cofibrant in $\C C$ then  $\phi$ is also a cofibration in $\C C$.
\end{cor}

\begin{cor}
Let $\C C$ be a model $\C V$-algebra satisfying  the strong unit axiom. 
Consider  an operad  $\mathcal O$  in ${\C V}$ such that $z(\mathcal{O}(n))$ is pseudo-cofibrant for all $n\geq 0$. Then, any cofibrant $\mathcal O$-algebra in $\C C$ has an underlying pseudo-cofibrant object. Moreover, if in addition $\mathcal O(0)$ is cofibrant then any cofibrant $\mathcal O$-algebra in $\C C$ has an underlying cofibrant object.
\end{cor}

The following theorem generalizes \cite[Theorem 1.3]{htnso} and Theorem \ref{coc} avobe.

\begin{thm}\label{cocfinal}
Let $\C C$ be a model $\C V$-algebra satisfying  the strong unit axiom. Assume $z\colon\C V\r\C C$ satisfies the $\unit$-cofibrant axiom. 
Consider a weak equivalence $\phi\colon \mathcal O\st{\sim}\r\mathcal P$ between operads in $\C V$ such that the objects $ \mathcal O(n)$ and $\mathcal P(n)$ are {($\unit$-)}cofibrant for all $n\geq 0$. Then the change of coefficients Quillen pair
$$\xymatrix{\algebra{\mathcal O}{\C{C}}\ar@<.5ex>[r]^-{\phi_{*}}&\algebra{\mathcal P}{\C{C}}\ar@<.5ex>[l]^-{\phi^{*}}}$$
is a Quillen equivalence. 
\end{thm}

Now, let us place ourselves in the situation described in the first  paragraph of Section \ref{malichon}. Assume further that $\C V$, $\C W$, $\C C$ and $\C D$ satisfy the strong unit axiom, and  $z_{\C C}$, $z_{\C D}$, $F$ and $\bar F$ satisfy the pseudo-cofibrant axiom and the $\unit$-cofibrant axiom.

\begin{lem}
The natural morphism $\tau(X)$ in \eqref{tau} is a weak equivalence if $X$ is ($\unit$-)cofibrant.
\end{lem}

Proposition \ref{paco2} extends as follows.

\begin{prop}\label{paco2final}
Suppose $\bar F\dashv \bar G$ is a weak monoidal Quillen adjunction and $\mathcal O$ is a cofibrant operad in $\C V$. Then if $A$ is a cofibrant $\mathcal O$-algebra in $\C C$, the natural morphism $\chi_{A}$ in \eqref{chia} is a weak equivalence in $\C D$.
\end{prop}

\begin{cor}\label{casigual2.1final}
Let $\bar F\dashv \bar G$ and $\mathcal O$ be as in the previous proposition. For any $\mathcal O$-algebra $A$ in $\C C$ there is a natural isomorphism in $\ho\C D$,
\begin{align*}
\mathbb{L}\bar F(A)&\cong\mathbb{L}\bar F_{\mathcal{O}}(A).
\end{align*}
\end{cor}

Proposition \ref{paco2final} allows to prove the following generalization of Theorem \ref{qea2}.

\begin{thm}\label{qea2final}
In the conditions of the previous proposition, suppose further that $\bar F\dashv \bar G$ is a Quillen equivalence.  Then the Quillen pair  $$\xymatrix{
\algebra{\mathcal O}{\C C}\ar@<.5ex>[r]^-{\bar F_{\mathcal O}}&\algebra{F^{\operatorname{oper}}(\mathcal O)}{\C D}
\ar@<.5ex>[l]^-{\bar G}
}$$
in Proposition \ref{qp2} is a Quillen equivalence. 
\end{thm}

Combining Theorems \ref{cocfinal} and \ref{qea2final} we can esily check that the following extension of Theorem \ref{qea3} holds. 

\begin{thm}\label{qea3final}
Suppose  $F\dashv G$ and $\bar F\dashv \bar G$ are weak (symmetric) monoidal Quillen equivalences, $\mathcal P$ is a fibrant operad in $\C W$, and $\mathcal P(n)$ and  $G(\mathcal P(n))$ are ($\unit$-)cofibrant for all $n\geq 0$. Then  the Quillen pair $$\xymatrix{
\algebra{G(\mathcal P)}{\C C}\ar@<.5ex>[r]^-{\bar F^{\mathcal P}}&\algebra{\mathcal P}{\C D}
\ar@<.5ex>[l]^-{\bar G}
}$$
 in Proposition \ref{qp3} is a Quillen equivalence. 
\end{thm}

\begin{cor}\label{casigual2.2final}
In the situation of the previous theorem,  for any $G(\mathcal P)$-algebra $B$ in $\C C$ there is a natural isomorphism in $\ho\C D$, 
\begin{align*}
\mathbb{L}\bar F(B)&\cong\mathbb{L}\bar F^{\mathcal{P}}(B).
\end{align*}
\end{cor}

We can similarly check  the following extension of Theorem \ref{qea4}.

\begin{thm}\label{qea4final}
Suppose  $\C V=\C W$, $F=G=1_{\C V}$,  $\bar F\dashv \bar G$ is a weak monoidal Quillen equivalence, and $\mathcal O$ is an operad in $\C V$ such that $\mathcal O(n)$ is ($\unit$-)cofibrant for all $n\geq 0$. Then  $$\xymatrix{
\algebra{\mathcal O}{\C C}\ar@<.5ex>[r]^-{\bar F_{\mathcal P}=\bar F^{\mathcal P}}&\algebra{\mathcal O}{\C D}
\ar@<.5ex>[l]^-{\bar G}
}$$ is a Quillen equivalence. 
\end{thm}

\begin{cor}\label{casigual2.3final}
Under the assumptions of the previous theorem, for any $\mathcal O$-algebra $B$ in $\C C$ there is a natural isomorphism in $\ho\C D$, 
\begin{align*}
\mathbb{L}\bar F(B)&\cong\mathbb{L}\bar F_{\mathcal{O}}(B)=\mathbb{L}\bar F^{\mathcal{O}}(B).
\end{align*}
\end{cor}

The following theorem generalizes Theorem \ref{lp}.

\begin{thm}\label{lp5final}
Let $\C C$ be a model $\C V$-algebra  satisfying the strong unit axiom. 
Consider an operad $\mathcal O$ in $\C V$ such that $z(\mathcal O(n))$ is pseudo-cofibrant for all $n\geq 0$, and a push-out diagram in $\algebra{\mathcal O}{\C{C}}$ as follows,
$$\xymatrix{A\ar@{ >->}[r]^-{\psi}\ar[d]_-{\varphi}^{\sim}\ar@{}[rd]|{\text{push}}&C\ar[d]^-{\varphi'}\\
B\ar@{ >->}[r]_-{\psi'}&B\cup_{A}C}$$
If the underlying objects of $A$ and $B$ are pseudo-cofibrant in $\C C$ then $\varphi'$ is a weak equivalence. 
\end{thm}

The following proposition is essentially a corollary of the previous theorem.

\begin{prop}
Let $\C C$ be a model $\C V$-algebra sastifying the strong unit axiom and $\mathcal O$ an operad in $\C V$ such that $z(\mathcal O(n))$ is pseudo-cofibrant for all $n\geq 0$. The full subcategory $\algebrapc{\mathcal O}{\C C}\subset \algebra{\mathcal O}{\C C}$ spanned by the $\mathcal O$-algebras whose underlying object in $\C C$ is pseudo-cofibrant is a cofibration category.
\end{prop}

In particular, the gluing lemma  also holds for $\algebrapc{\mathcal O}{\C C}$ in these circumstances. %, compare Corollary \ref{gluing}.

\section{Generalized results for $A$-infinity algebras}\label{genainf}

We finish this paper by indicating how the results in Section \ref{ainf} extend when tensor units are not cofibrant. Notice that the (unital) associative operad consists of ($\unit$-)cofibrant objects in all arities.

\begin{thm}
Let $F\colon\C{V}\rightleftarrows\C{W}\colon G$ be a weak symmetric monoidal Quillen adjunction between symmetric monoidal model categories satisfying  the strong unit axiom. Assume $F\dashv G$ satisfies the pseudo-cofibrant axiom and the $\unit$-cofibrant axiom. 
% Take cofibrant resolutions
% \begin{align*}
% q\colon \mathtt{uA}^{\C V}_{\infty}&\st{\sim}\onto \mathtt{uAss}^{\C V},&
% q'\colon \mathtt{A}^{\C V}_{\infty}&\st{\sim}\onto \mathtt{Ass}^{\C V},
% \end{align*}
% in $\operad{\C V}$. 
Then the morphisms
\begin{align*}%\label{assad}
\psi\colon F^{\operatorname{oper}}(\mathtt{uA}^{\C V}_{\infty})&\To \mathtt{uAss}^{\C W},&
\psi'\colon F^{\operatorname{oper}}(\mathtt{A}^{\C V}_{\infty})&\To \mathtt{Ass}^{\C W},
\end{align*}
adjoint to $\varphi q$ and $\varphi'q'$, respectively,  are weak equivalences in $\operad{\C W}$.
\end{thm}

In particular, Theorem \ref{invariance} and Corollary \ref{invariance2} hold under these weaker hypotheses.

\begin{thm}\label{invariancefinal}
Under the hypotheses of the previous theorem, there are isomorphisms in $\ho\operad{\C{W}}$,
\begin{align*}
\mathbb L F^{\operatorname{oper}}(\mathtt{Ass}^{\C{V}})&\cong \mathtt{Ass}^{\C{W}},&
\mathbb L F^{\operatorname{oper}}(\mathtt{uAss}^{\C{V}})&\cong \mathtt{uAss}^{\C{W}},
%\\
%\mathbb R G^{\operatorname{oper}}\mathtt{Ass}^{\C{W}}&\cong \mathtt{Ass}^{\C{V}},&
%\mathbb R G^{\operatorname{oper}}\mathtt{uAss}^{\C{W}}&\cong \mathtt{uAss}^{\C{V}},
\end{align*}
such that the following square commutes,
$$\xymatrix@C=50pt{
\mathbb L F^{\operatorname{oper}}(\mathtt{Ass}^{\C{V}})\ar[r]^{\mathbb L F^{\operatorname{oper}}(\phi^{\C V})}\ar[d]_{\cong}&
\mathbb L F^{\operatorname{oper}}(\mathtt{uAss}^{\C{V}})\ar[d]^{\cong}\\
\mathtt{Ass}^{\C{W}}\ar[r]^{\phi^{\C W}}&
\mathtt{uAss}^{\C{W}}
}
%\qquad
%\xymatrix@C=30pt{
%\mathbb R G\mathtt{Ass}^{\C{W}}\ar[r]^{\mathbb R G\phi^{\C W}}\ar[d]_{\cong}&
%\mathbb R G\mathtt{uAss}^{\C{W}}\ar[d]^{\cong}\\
%\mathtt{Ass}^{\C{V}}\ar[r]^{\phi^{\C V}}&
%\mathtt{uAss}^{\C{V}}
%}
$$
\end{thm}

\begin{cor}\label{invariance2final}
Under the hypotheses of the previous theorem, if in addition $F\dashv G$ is a Quillen equivalence, then there are isomorphisms in $\ho\operad{\C{V}}$,
\begin{align*}
%\mathbb L F^{\operatorname{oper}}\mathtt{Ass}^{\C{V}}&\cong \mathtt{Ass}^{\C{W}},&
%\mathbb L F^{\operatorname{oper}}\mathtt{uAss}^{\C{V}}&\cong \mathtt{uAss}^{\C{W}},
%\\
 \mathtt{Ass}^{\C{V}}&\cong\mathbb R G(\mathtt{Ass}^{\C{W}}),&
\mathtt{uAss}^{\C{V}}&\cong \mathbb R G(\mathtt{uAss}^{\C{W}}),
\end{align*}
such that the following square commutes,
$$
%\xymatrix@C=40pt{
%\mathbb L F^{\operatorname{oper}}\mathtt{Ass}^{\C{V}}\ar[r]^{\mathbb L F^{\operatorname{oper}}\phi^{\C V}}\ar[d]_{\cong}&
%\mathbb L F^{\operatorname{oper}}\mathtt{uAss}^{\C{V}}\ar[d]^{\cong}\\
%\mathtt{Ass}^{\C{W}}\ar[r]^{\phi^{\C W}}&
%\mathtt{uAss}^{\C{W}}
%}
%\qquad
\xymatrix@C=40pt{
\mathtt{Ass}^{\C{V}}\ar[r]^{\phi^{\C V}}\ar[d]_{\cong}&
\mathtt{uAss}^{\C{V}}\ar[d]^{\cong}\\
\mathbb R G(\mathtt{Ass}^{\C{W}})\ar[r]^{\mathbb R G(\phi^{\C W})}&
\mathbb R G(\mathtt{uAss}^{\C{W}})
}
$$
\end{cor}

The following result is a consequence of Theorem  \ref{cocfinal}.

\begin{prop}\label{coc1.1final}
Let $\C C$ be a model $\C V$-algebra. Suppose $\C V$ and $\C C$ satisfy  the strong unit axiom. 
Assume further that $z\colon\C V\r\C C$ satisfies the pseudo-cofibrant axiom and the $\unit$-cofibrant axiom. 
Then, $q$ and $q'$
% cofibrant resolutions
% \begin{align*}
% q\colon \mathtt{uA}^{\C V}_{\infty}&\st{\sim}\onto \mathtt{uAss}^{\C V},&
% q'\colon \mathtt{A}^{\C V}_{\infty}&\st{\sim}\onto \mathtt{Ass}^{\C V},
% \end{align*}
% in $\operad{\C V}$ 
induce Quillen equivalences
$$\xymatrix{\algebra{\mathtt{uA}^{\C V}_{\infty}}{\C{C}}\ar@<.5ex>[r]^-{q_{*}}&\algebra{\mathtt{uAss}^{\C V}}{\C{C}},\ar@<.5ex>[l]^-{q^{*}}}\qquad
\xymatrix{\algebra{\mathtt{A}^{\C V}_{\infty}}{\C{C}}\ar@<.5ex>[r]^-{q_{*}'}&\algebra{\mathtt{Ass}^{\C V}}{\C{C}}.\ar@<.5ex>[l]^-{(q')^{*}}}$$
\end{prop}

Let us place ourselves in the situation described in the first  paragraph of Section \ref{malichon}. Assume further that $\C V$, $\C W$, $\C C$ and $\C D$ satisfy  the strong unit axiom, and  $z_{\C C}$, $z_{\C D}$, $F$ and $\bar F$ satisfy the pseudo-cofibrant axiom and the $\unit$-cofibrant axiom.

\begin{cor}\label{qea3.1final}
Suppose  $F\dashv G$ and $\bar F\dashv \bar G$ are weak (symmetric) monoidal Quillen equivalences. Then  there are Quillen equivalences $$\xymatrix{
\algebra{\mathtt{uAss}^{\C V} }{\C C}\ar@<.5ex>[r]&\algebra{\mathtt{uAss}^{\C W} }{\C D}
\ar@<.5ex>[l]^-{\bar G},
}\qquad 
\xymatrix{
\algebra{\mathtt{Ass}^{\C V} }{\C C}\ar@<.5ex>[r]&\algebra{\mathtt{Ass}^{\C W} }{\C D}
\ar@<.5ex>[l]^-{\bar G}.
}$$ 
\end{cor}

This shows that Corollaries \ref{invariance3} and  \ref{invariance4} are true provided $\C V$ and $\C W$ satisfy  the strong unit axiom and $F$ satisfies the pseudo-cofibrant axiom and the $\unit$-cofibrant axiom.

% ----------------------------------------------------------------
%\bibliographystyle{amsalpha}
%\bibliography{../Fernando}

\providecommand{\bysame}{\leavevmode\hbox to3em{\hrulefill}\thinspace}
\providecommand{\MR}{\relax\ifhmode\unskip\space\fi MR }
\providecommand{\MRhref}[2]{%
  \href{http://www.ams.org/mathscinet-getitem?mr=#1}{#2}
}
\providecommand{\href}[2]{#2}

\end{document}